\documentclass{ws-m3as}
\usepackage{amsmath,amssymb,amsfonts}
\usepackage[greek,english]{babel}
\usepackage{verbatim}
\usepackage{graphicx}
\usepackage[active]{srcltx}
\usepackage{enumerate}

\usepackage{tikz}
\usepackage{pgfplots}

\makeatletter
\renewcommand\theequation{\thesection.\arabic{equation}}
\@addtoreset{equation}{section}
\makeatother
\newcommand{\disp}{\displaystyle}

\newcommand{\Om}{\Omega}
\newcommand{\be}{\begin{equation}}
\newcommand{\ee}{\end{equation}}
\newcommand{\bestar}{\begin{equation*}}
\newcommand{\eestar}{\end{equation*}}
\newcommand{\tOm}{\tilde{\Omega}}

\newcommand{\ra}{\rightarrow}
\newcommand{\rau}{\rightharpoonup}

\newcommand{\R}{\mathbb R}
\newcommand{\N}{\mathbb N}
\newcommand{\Z}{\mathbb Z}
\newcommand{\T}{\mathbb T}

\newcommand{\Kc}{\mathcal{K}}
\newcommand{\Mc}{\mathcal{M}}

\newcommand{\pa}{\partial}
\newcommand{\na}{\nabla}
\newcommand{\dv}{\mathrm{div}\;}

\newcommand{\eps}{\epsilon}

\newcommand{\aeps}{a_\epsilon}

\newcommand{\bu}{\bar u}

\newcommand{\ue}{u^\epsilon}
\newcommand{\pe}{p^\epsilon}

\newcommand{\tB}{\tilde{B}}
\newcommand{\Vepsinf}{V^{\eps, \infty}}
\newcommand{\phiseps}{\phi_s^\eps}
\newcommand{\Tc}{\mathcal{T}}

\def\d{\partial}

\newcommand{\ud}{\mathrm{d}}

\def\div{\hbox{div  }}

\title{COMPUTATION OF THE EFFECTIVE SLIP\\ OF ROUGH HYDROPHOBIC SURFACES \\ {\it VIA} HOMOGENIZATION}

\begin{document}
\bibliographystyle{amsplain}

\markboth{Bonnivard, Dalibard, G\'erard-Varet}{Computation of the effective slip of rough hydrophobic surfaces {\it via} homogenization}

%%%%%%%%%%%%%%%%%%% Publisher's Area please ignore %%%%%%%%%%%%%%%%%%%%%%%
%
\catchline{}{}{}{}{}
%
%%%%%%%%%%%%%%%%%%%%%%%%%%%%%%%%%%%%%%%%%%%%%%%%%%%%%%%%%%%%%%%%%%%%%%%%%%

\author{MATTHIEU BONNIVARD}

\address{Universit\'e Paris Diderot, Sorbonne Paris Cit\'e, Laboratoire Jacques-Louis Lions,\\ UMR 7598, UPMC, CNRS, F-75205 Paris, FRANCE\\ bonnivard@math.univ-paris-diderot.fr}

\author{ANNE-LAURE DALIBARD}
\address{DMA/CNRS, \'Ecole Normale Sup\'erieure,\\ 45 rue d'Ulm, 75005
   Paris, FRANCE\\
   Anne-Laure.Dalibard@ens.fr
   }

\author{DAVID G\'ERARD-VARET}
\address{Institut de Math\'ematiques de Jussieu, Universit\'e Paris Diderot,\\ UFR de Math\'ematiques 
B\^atiment Sophie Germain, \\
75205 Paris Cedex 13,  FRANCE\\
gerard-varet@math.jussieu.fr}

\maketitle

\begin{history}
\received{(Day Month Year)}
\revised{(Day Month Year)}
%\accepted{(Day Month Year)}
\comby{(xxxxxxxxxx)}
\end{history}

\begin{abstract}
We present a quantitative analysis of the effect of rough hydrophobic surfaces on viscous newtonian flows. 
We use a model introduced by Ybert and coauthors in Ref. \refcite{Ybert}, in which the rough surface is replaced by a flat plane with alternating small areas of slip and no-slip. We investigate the averaged slip generated at the boundary, depending on  the ratio between these areas. 
This problem reduces to  the homogenization of a non-local system, involving the Dirichlet to Neumann map of the Stokes operator, in a domain with small holes. Pondering on the works of Allaire (see Ref. \refcite{Allaire1,Allaire2}) we  compute accurate scaling laws of the averaged slip for various types of roughness (riblets, patches).  Numerical computations complete and confirm the analysis.  
\end{abstract}



\section{Introduction}
With the development of microfluidics, drag reduction for low Reynolds number flows, notably at  solid walls, has become a stimulating issue. Therefore,   the interaction between a fluid and a solid boundary has been investigated thoroughly, both at the experimental and theoretical levels. A special attention has been paid to the detection of slip, for various types of flows and solid walls. We refer to Ref. \refcite{Lauga} for a review.

\medskip
 As a result of this activity, the idea that {\em rough boundaries could generate  a substantial slip} has spread out.  This idea has developped on the basis of  both experimental and theoretical works, notably on {\em wall laws}.  We remind that in the context of roughness effects, a wall law is an effective boundary condition imposed at a smoothened boundary, reflecting the overall impact of the real rough boundary. 
In particular, if one describes the rough boundary through an oscillation of small amplitude and wavelength $\eps$, one can show rigorously  that a no-slip condition at the rough boundary can be replaced by a wall law of Navier type, with slip length of order $\eps$. We refer for instance to articles \refcite{AcPiVa,JaMi,BaGe} for more precise statements.  

\medskip
However, these seemingly favorable results must be considered with care. For instance, at the experimental level, one must ensure that the slip is not measured too far away from the boundary.  Also, as regards the theoretical works on wall laws, {\em the position of the artificial boundary at which the law is prescribed is crucial}. Indeed, when the artificial boundary is moved upwards by a height $h=O(\eps)$, the effective slip is also increased by $h$. Let us emphasize that all forementioned works consider artifical boundaries that are at the top of the roughness. As a result, the flow rate in the smoothened domain does not equal the averaged flow rate in the rough domain, making comparisons inaccurate.  In fact, in  the case of rough wetting surfaces (endowed with a no-slip condition), one can even show the following: if one puts  the artificial boundary in a way that the flow rates are the same, {\em then the flat boundary is optimal with respect to drag minimization}. We refer to Ref. \refcite{BuDaGe} for detailed statements and proofs. Hence, the possibility of decreasing drag through roughness is not so clear, especially for rough wetting surfaces.

\medskip
Still, in the recent years, promising results have been obtained concerning a class of rough hydrophobic surfaces, see for instance Ref. \refcite{Vino}. Indeed, by the combination of the chemical and geometrical properties of these surfaces, the hollows of the roughness get filled with gas. Hence, the viscous fluid above does not penetrate: it slips above the hollows, and only sticks at the bumps, reaching the so-called {\em Cassie or fakir state}. 

\medskip
The aim of this paper is to study the slip generated by such configurations, both in a rigorous and quantitative manner. 
We focus on a model proposed in article \refcite{Ybert}, in which the rough boundary is replaced by a flat plane, divided in small periodic cells (say of side $\eps \ll 1$). Each cell is divided in two zones: 
\begin{itemize}
\item A no-slip zone, corresponding to a plane projection of the sticky part of the roughness (bumps). 
\item A slip-zone, corresponding to a plane projection of the slippery part.
\end{itemize}
Using  homogenization techniques, we derive an effective boundary condition as $\eps$ goes to zero, depending on the characteristic scale $\aeps$ of the no-slip zones. We provide in this way scaling laws for the slip coefficients, for various configurations (patches, riblets). 
Such laws are  in global agreement with the formal computations led in Ref. \refcite{Ybert}. {\em One shows notably that the riblet configuration is less effective than patches one (see Remark \ref{rem:patches-vs-riblets}).} All our theoretical results are grounded by numerical computations at the end of the paper.

\section{Main results}

Let us first present the model under study. We consider a three-dimensional Stokes flow between two infinite plates:
\be
\label{eq:Stokes}
\begin{aligned}
-\Delta u + \na p =f\quad  \text{in }\Om,\\
\dv u =0\quad \text{in }\Om,
\end{aligned}
\ee
where $\Om= \T^2\times (0,1)$ and $\T^2= \R^2/\Z^2$. We denote by $x = (x_1,x_2,x_3) = (x_h,x_3)$ the space variable. The function $f\in L^2(\Om)$ is a given source term.
On the upper surface $x_3=1$, we enforce a ``no-slip'' boundary condition
\be\label{noslip-z=1}
u_{|x_3=1}=0.
\ee
On the lower surface, we assume that $u$ satisfies alternately ``perfect slip'' and ``no slip'' boundary conditions, corresponding respectively to the hollows and bumps of the rough hydrophobic surface. More precisely, let $\eps > 0$ and 
$$S^\eps := [0,\eps)^2 \: \sim \: \left( \R/\left(\eps \Z\right) \right)^2,  $$
 the elementary  square of side~$\eps \:$. For simplicity, we shall  assume all along that $\eps^{-1}$ is an integer.  Let $T^\eps$ be a Lipschitz subdomain of $S^\eps$, modeling an elementary no-slip zone.  Details about $T^\eps$ will be given right below. 
From this  elementary no-slip zone, we define  a global one inside $[0,1)^2 \sim \T^2$:
$$ {\cal T}^\eps \: := \:   \bigcup_{\substack{k\in [|0, \dots, \eps^{-1}|]^2}} \left(\eps k + T^\eps\right). $$
Finally, the boundary condition at $x_3=0$ is
\be 
\label{mixed-z=0}
u_{3|x_3=0}=0,\quad 
\pa_{3} u_{h|x_3=0}=0\text{ on }({\cal T}^\eps)^ c \times \{0\}, \quad u_{h|x_3=0}=0\text{ on }{\cal T}^\eps \times \{0\}.
\ee
 It is easily proved that \eqref{eq:Stokes}-\eqref{noslip-z=1}-\eqref{mixed-z=0} has a unique solution $(\ue,\pe)\in H^1(\Om)\times L^2(\Om)/\R$.
 
 \medskip
This article is devoted to the asymptotic analysis  of $(\ue,\pe)$,  as $\eps\to 0$. We will distinguish between two  types of no-slip pattern  $T^\eps$:
 \begin{itemize}
 \item \textbf{Patches:}  we assume that
\be\label{DefPatches}
T^\eps \: := \:  \begin{pmatrix}
\displaystyle \frac{\eps}{2}\\~ \\\displaystyle \frac{\eps}{2}
\end{pmatrix}  + \aeps T, 
\ee
where  $\begin{pmatrix}
\frac{\eps}{2}\\\frac{\eps}{2}
\end{pmatrix} $ is the center of the square $S^\eps$, and where the  domain $T$ is relatively compact  in the square $\displaystyle (-1/2,1/2)^2$, and contains a disk of radius $\alpha>0$, centered in the origin (see Figure \ref{Fig:T-Teps}). The parameter $\aeps$ is a positive number such that $\aeps<\eps$. In this case, the no-slip zone is a union of periodically distributed patches.

 \item \textbf{Riblets:}  we assume that  
 \be \label{DefRiblets}
 T^\eps \: := \:   (\eps \T) \times  \left(\frac{\eps}{2} + \aeps I\right).  
 \ee
 where $I\subset  (-\frac{1}{2},\frac{1}{2})$ is an  open interval (see Figure \ref{Fig:Riblets}).   
In this case, the no-slip zone is a union of stripes, invariant in the $x_1$-direction. Of course, invariance in the $x_2$ direction could have been considered as well. 
Note that later on,  addressing the case of riblets,  we shall focus on two particular cases:
\begin{itemize}
\item $f=e_1$: riblets parallel to the flow;
\item $f=e_2$: riblets perpendicular to the flow.
\end{itemize}

%%%%%%%%%%%%%%%%%%%%%%%%%%%%%%%%%%%
%%%%%% Figures pour patch et riblet %%%%%%%%
%%%%%%%%%%%%%%%%%%%%%%%%%%%%%%%%%%%

\begin{figure} 
	\begin{minipage}[t]{200pt}
		\includegraphics[scale=0.20]{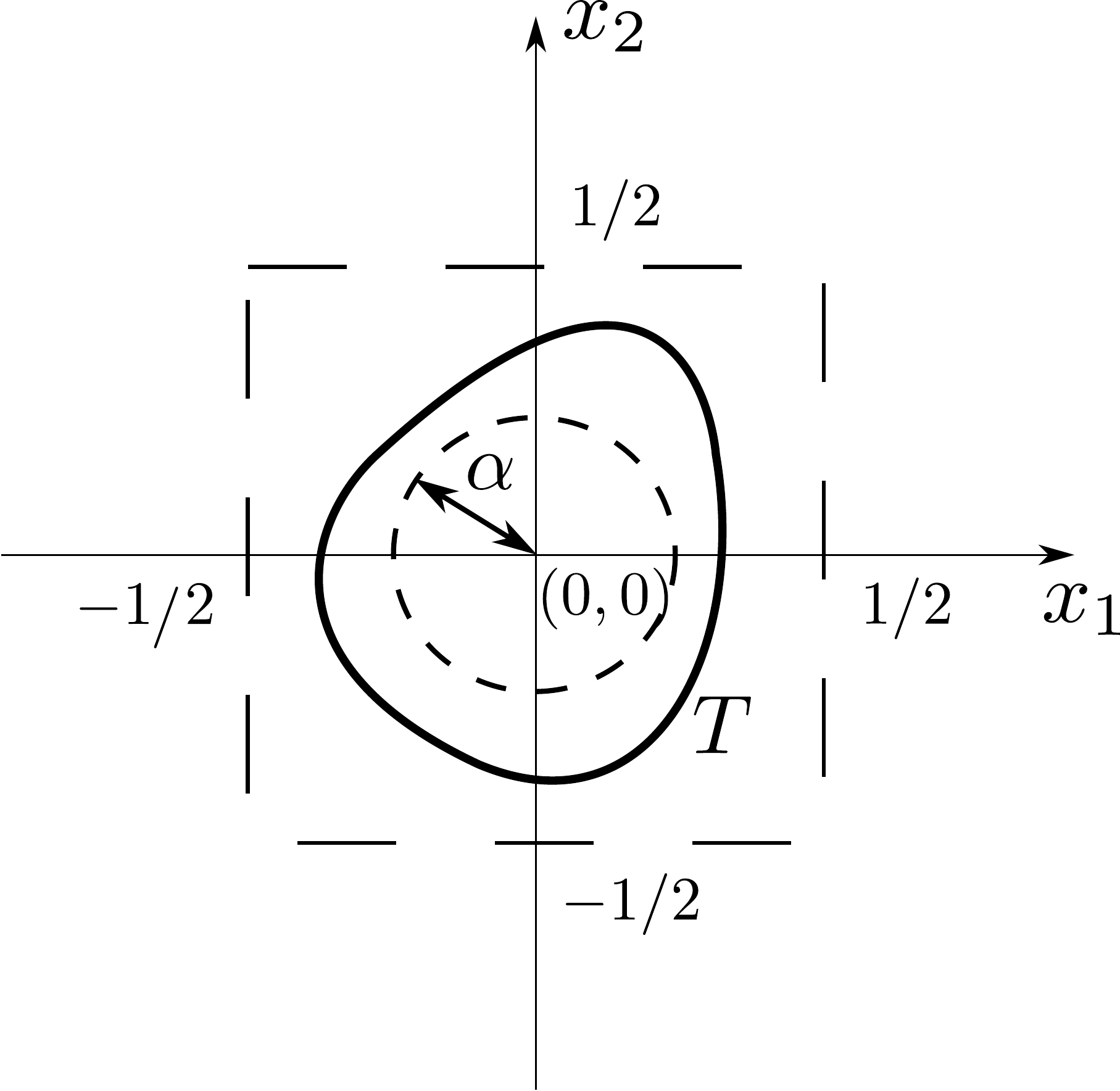}
	\end{minipage}
	\begin{minipage}[t]{160pt}
		\includegraphics[scale=0.20]{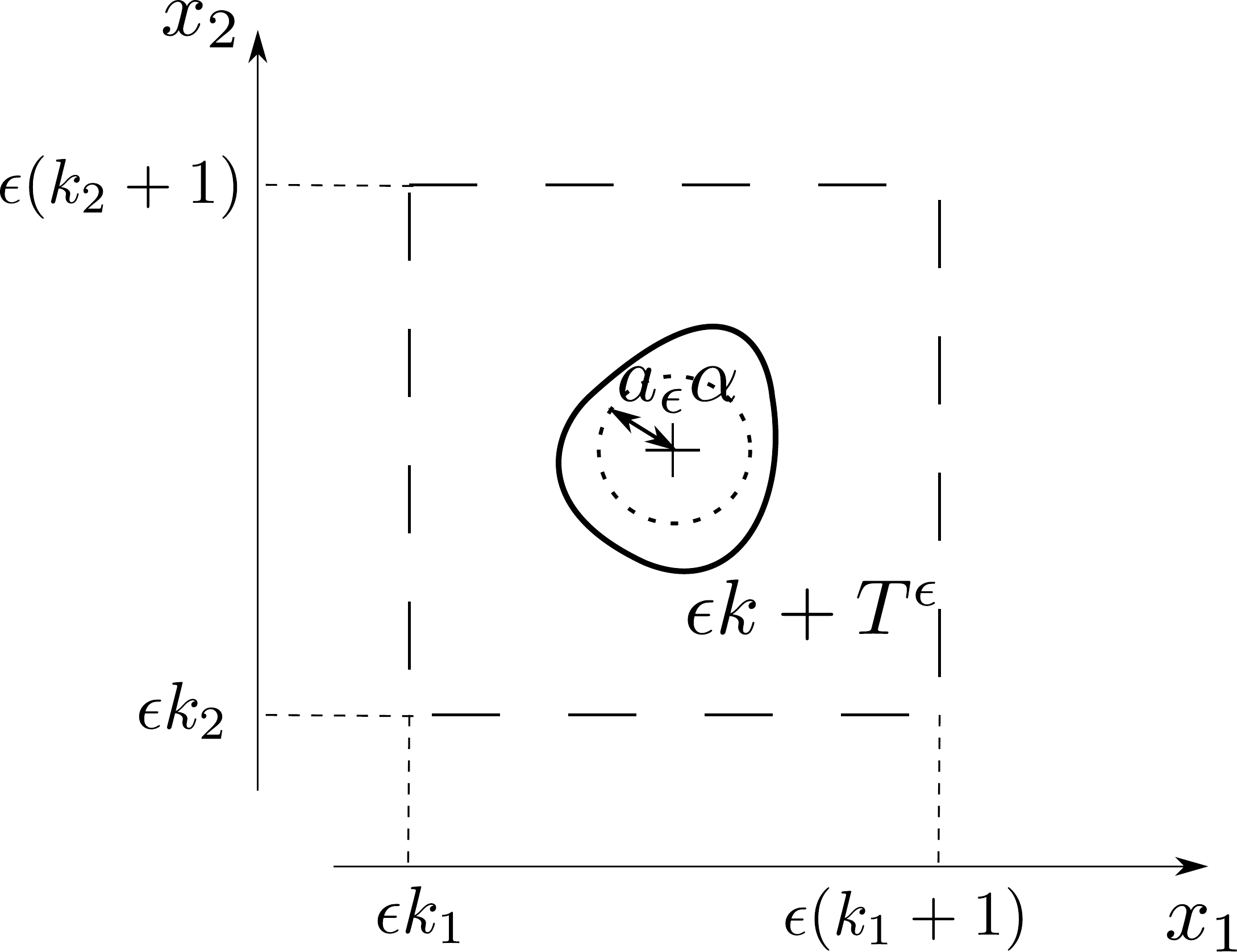}
	\end{minipage}
\caption{Patch configuration. For every $k=(k_1,k_2)\in [|0, \dots, \eps^{-1}|]^2$,
the intersection of the no-slip zone $\Tc^\eps$
with the cell $[\eps k_1,\eps(k_1+1))\times [\eps k_2,\eps(k_2+1))$ is defined by $\eps k + T^\eps = \eps k + \aeps T$.
 }\label{Fig:T-Teps}
\end{figure}

\begin{figure} 
\begin{center}
	\includegraphics[scale=0.20]{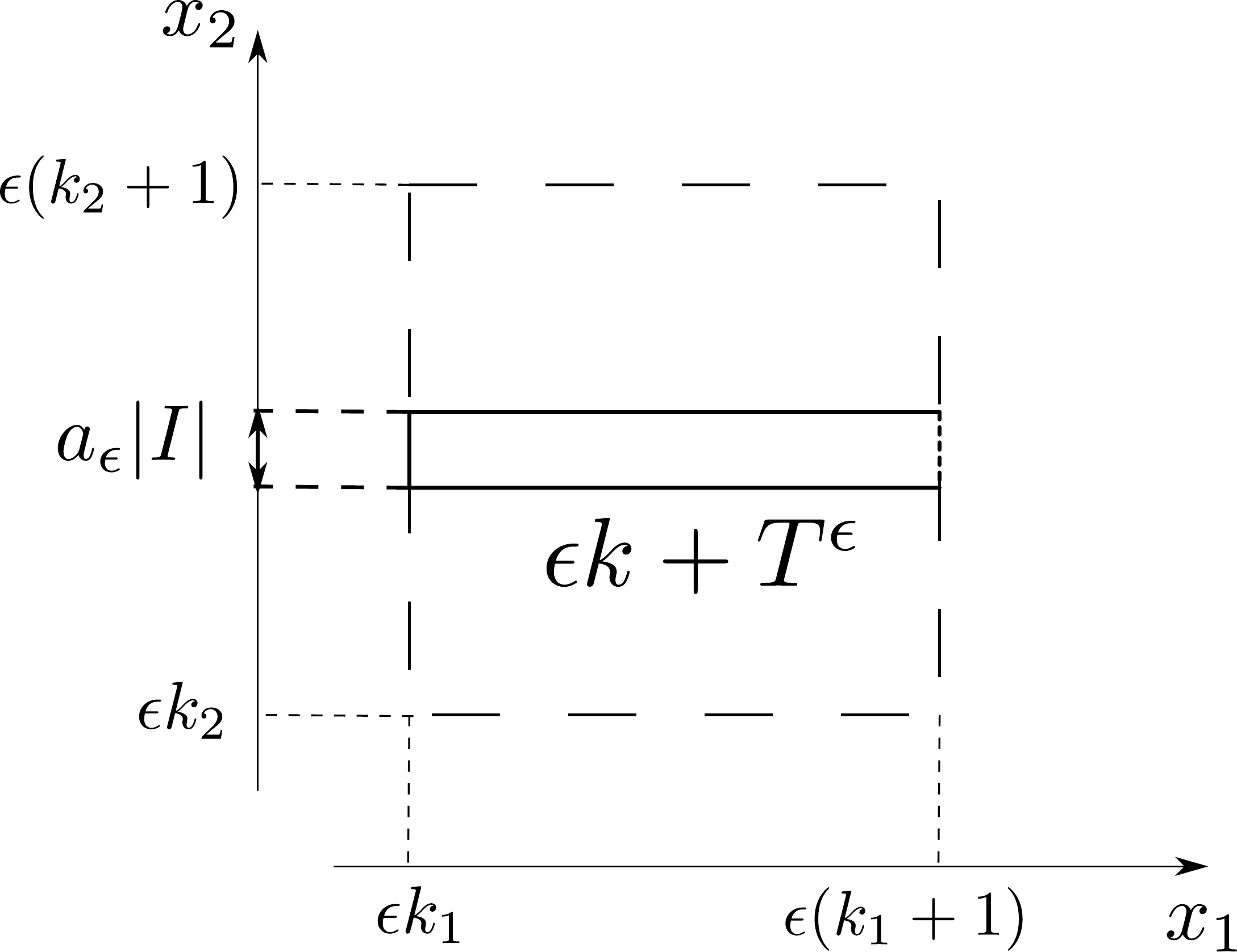}
\end{center}
\caption{Riblet configuration. For $k=(k_1, k_2)$, the intersection of the no-slip zone $\Tc^\eps$
with the cell $[\eps k_1,\eps(k_1+1))\times [\eps k_2,\eps(k_2+1))$ is defined by $\eps k + T^\eps = \eps k + (\eps \T) \times   (\frac{\eps}{2} + \aeps I)$. } \label{Fig:Riblets}
\end{figure}

  \end{itemize}
The issue is to derive a wall law for the system \eqref{eq:Stokes}-\eqref{noslip-z=1}-\eqref{mixed-z=0}, that is, to replace the mixed boundary condition \eqref{mixed-z=0} at $x_3=0$ by a condition which does not depend on $\eps$. We will show that $\ue$ behaves asymptotically like  the solution $\bu$  in $H^1$ of \eqref{eq:Stokes}-\eqref{noslip-z=1}, endowed either with a Navier boundary condition
\be\label{wall-law-Navier}
u_3=0\text{ at } x_3=0,
\quad  \pa_3 u_h  = M u_h  \text{ at } x_3=0,
\ee
or with a Dirichlet boundary condition
\be\label{wall-law-Dirichlet}
u_{|x_3=0}=0.
\ee
In \eqref{wall-law-Navier}, $M$ is a $2\times 2$ non-negative matrix, whose eigenvalues have the dimension of the inverse of a length. If $M=\lambda \mathrm{Id}$, the number $\lambda^{-1}$ is called the ``slip length''. In the general case, the inverse of the eigenvalues provide the slip lengths in the directions of the eigenvectors.  We shall denote  $\bar{u}_M$ the solution of \eqref{eq:Stokes}-\eqref{noslip-z=1}-\eqref{wall-law-Navier}. We will write  $\bar{u}_0$ in the special case $M = \left(\begin{smallmatrix} 0  & 0 \\ 0 & 0 \end{smallmatrix} \right)$. Eventually, we shall denote $\bar{u}_\infty$ the solution of \eqref{eq:Stokes}-\eqref{noslip-z=1}-\eqref{wall-law-Dirichlet}.

\medskip
With the previous notation, we can state our first result: 
\begin{theorem} (Asymptotic behavior for patches) 

\smallskip
Assume that $T^\eps \: := \:   (\frac{\eps}{2},\frac{\eps}{2})  + \aeps T$, where $T\Subset (-1/2,1/2)^2$ contains a disc of radius $\alpha>0$ centered in the origin. Let $\ue\in H^1(\T^2\times(0,1))$ be the solution of \eqref{eq:Stokes}, \eqref{noslip-z=1}, \eqref{mixed-z=0}.  One must distinguish between three cases:   

\begin{enumerate}
\item Sub-critical case: if $\aeps\ll \eps^2$, then $\ue\rightharpoonup \bu_0$ in $H^1(\T^2\times(0,1))$;
\item Super-critical case: if $\aeps\gg \eps^2$, then $\ue\rightharpoonup \bu_\infty$ in $H^1(\T^2\times(0,1))$; 
\item Critical case: there exists a symmetric, positive definite matrix $M_0$ such that if ${\aeps}/{\eps^2}\to C_0>0$, then $\ue\rightharpoonup \bu_{C_0M_0}$ in $H^1(\T^2\times(0,1))$.

\end{enumerate}

\label{thm:patch}
\end{theorem}

A similar result holds for riblets. Let us merely state the theorem in the critical case:
\begin{theorem} (Asymptotic behaviour for riblets)

\smallskip
Assume that $T^\eps \: := \:   (\eps \T) \times  (\aeps I)$, where $I\subset(-1/2,1/2)$ is an open interval.   Suppose that $\lim_{\eps\to 0} -\eps\ln(\aeps)=C_0>0$, and furthermore that $f$ does not depend on $x_1$.

\smallskip
Then, $\ue\rightharpoonup \bu_{M_{rib}}, \: $ where
\be \label{Mparallel}
M_{rib}= \begin{pmatrix}
\frac{\pi}{C_0}&0\\0&\frac{2\pi}{C_0}
\end{pmatrix}. 
\ee
Additionally, when $f=e_1$ or $f=e_2$, the limit system can be simplified:
\begin{itemize}
\item if $f=e_1$ (riblets parallel to the main flow), then $\bu_{M_{rib},2}=\bu_{M_{rib},3}=0$ and $\bu_{M_{rib},1}$ satisfies
$$
\pa_3 \bu_{M_{rib},1}= \frac{\pi}{C_0} \bu_{M_{rib},1}\quad \text{at } x_3=0.
$$
Hence, the slip length is $C_0/\pi$;

\item if $f=e_2$ (riblets perpendicular to the main flow), then $\bu_{M_{rib},1}=0$ and $\bu_{M_{rib},2}$ satisfies
$$
\pa_3 \bu_{M_{rib},2}= \frac{2\pi}{C_0} \bu_{M_{rib},2}\quad \text{at } x_3=0.
$$
Hence, the slip length is $C_0/(2\pi)$.
\end{itemize}

\label{thm:riblet}
\end{theorem}

\begin{remark}
Notice that in the critical and supercritical cases, the slip length is respectively of order one and infinite in the limit. Therefore large slip is achieved in the limit, which differs  from previous papers on the subject (see Ref. \refcite{JaMi,AmSi}).
\end{remark}

\begin{remark}\label{rem:patches-vs-riblets}
%Notice that theorems \ref{thm:patch} and \ref{thm:riblet} indicate that for a given  area of  no-slip, riblets are less efficient than patches in terms of slip optimization. Indeed, consider riblets $T^\eps$  with  area  $\eps \aeps$ corresponding to the critical case in Theorem \ref{thm:riblet}  $\lim_{\eps\to 0} -\eps\ln(\aeps)=C_0>0$. Patches with the same area have a diameter of order $\sqrt{\eps\aeps}$:  they have a sub-critical behavior (with regards to Theorem \ref{thm:patch}) since $\sqrt{\eps\aeps}\ll \eps^2$. So, $u^\eps$ converges towards $\bu_0$. In other words, for the same area, patches achieve perfect slip on the lower surface, while riblets only have finite slip. In a similar fashion, it can be proved that riblets whose area $\aeps^2$ is critical for patches   - i. e. such that $\aeps \sim \eps^2$, see Theorem \ref{thm:patch}- have a super-critical behavior (with regards to Theorem \ref{thm:riblet}) and converge towards $\bu_\infty$: the solution satisfies a no-slip condition at the limit. 
%%Notice that in the latter case, each $\Te_k$ takes the form $\Te_k=c_k^\eps +  (\eps \T)\times(b^\eps I)$, with  $b^\eps=\aeps^2/\eps$.

\end{remark}

\begin{remark}
Our results are consistent with those of Ref. \refcite{Ybert}: indeed, in the case of patches, it is shown heuristically there  that the slip length is proportional to $\eps^2/\aeps$: in other words, if $\aeps\ll \eps^2$, perfect slip is achieved, if $\aeps\gg \eps^2$,  a no-slip condition is retrieved in the limit, and in the critical case, the slip length is positive and finite.

\medskip
Also, explicit calculations (see Ref. \refcite{Phi}) recalled in Ref. \refcite{Ybert} show that the slip length for riblets is equal to $-\eps/\pi \ln (\aeps/\eps)$ for riblets parallel to the flow, and to $-\eps/(2\pi) \ln (\aeps/\eps)$ for riblets perpendicular to the flow. Once again, this is consistent with Theorem \ref{thm:riblet}.

\end{remark}

\begin{remark}
Theorems 1 and 2 do not support the idea that rough hydrophobic surfaces can generate  a substantial slip. Indeed, to obtain an effective slip law, the surface fraction of no-slip has to be very small. Back to wall roughness, this would correspond to  narrow peaks separated by (comparatively) large hollows. It seems far from the roughness characteristics used experimentally to obtain a hydrophobic Cassie state.     
\end{remark}

The proofs of theorems 1 and 2, that rely strongly on the papers  \refcite{Allaire1,Allaire2} by Allaire, are given in Section \ref{sec:patch} and \ref{sec:riblet} respectively. We  then present in Section \ref{sec:simul} numerical simulations that confirm the asymptotic results, and clarify the influence of the shape of patches on the slip length, i.e. on the eigenvalues of the matrix $M$.

%%%%%%%%%%%%%%%%%%%%%%%%%%%%%%%
%%%%% PREUVES CAS DES PATCHS %%%%%%
%%%%%%%%%%%%%%%%%%%%%%%%%%%%%%%

\section{Asymptotic study of ``patch'' designs}
\label{sec:patch}

This section is devoted to the proof of Theorem \ref{thm:patch}. Let $(\ue,\pe)\in H^1(\Om)\times L^2(\Om)/\R$ be the solution of \eqref{eq:Stokes}, \eqref{noslip-z=1}, \eqref{mixed-z=0}. By classical arguments, the sequence $(\ue,\pe)$ is uniformly bounded in $H^1(\Om)\times L^2(\Om)/\R$, and consequently, there exists a couple $(\overline{u},\overline{p})\in H^1(\Om)\times L^2(\Om)/\R$ such that
\bestar
u^\eps \rau \overline{u} \quad \textrm{weakly in }H^1(\Om),\quad p^\eps \rau \overline{p} \quad \textrm{weakly in }L^2(\Om)/\R.
\eestar
Using the weak formulation of Eqs.\,\eqref{eq:Stokes} and the continuity of the trace operator, one obtains easily that the weak-limit $(\overline{u},\overline{p})$ satisfies Eqs.\,\eqref{eq:Stokes} and boundary condition \eqref{noslip-z=1} on $x_3=1$. On $x_3=0$, the boundary condition satisfied by the vertical component is preserved in the limit, and we obtain
$
\overline{u}_{3|x_3=0}=0.
$
 To describe the boundary condition satisfied by the horizontal components $(\overline{u}_1,\overline{u}_2)$ on $x_3=0$, we need to distinguish between the so-called super-critical, critical and sub-critical cases.
 
\medskip 
 
\textbf{Notation.}  For every $k=(k_1,k_2)\in [|0,\eps^{-1}|]^2$, we 
denote the elementary  squares,  cubes and half-cubes as follows:
\be
S^\eps_k \: := \: \eps k + S^\eps, \quad    P_k^\eps \:  := S^\eps_k \times (-\frac{\eps}{2},\frac{\eps}{2}),  \label{def:cube-Pk}  \quad
P_k^{\eps,+} \: := \:  P_k^\eps \cap \R_+^3.
\ee
We shall use that notation throughout the paper. 
%We denote by  $B_k^\eps$  the ball circumscribing the cube $P_k^\eps$ and $B_k^{\eps,+}$  the corresponding half ball, namely $B_k^{\eps,+} \: := \: %B_k^\eps\cap \R_+^3$.

\medskip
\textbf{Super-critical case: $\aeps\gg \eps^2$.}
The proof in the super-critical case relies on a quantitative Poincar\'e inequality: we claim that there exist $\eps_0 > 0$ and a positive function $\eta(\eps)$ such that $\eta(\eps)\ra 0$ as $\eps\ra 0$, and such that
\be\label{ineq:super-critical} 
\int_{\T^2\times\{0\}} |u^{\epsilon}|^2 \leq \eta(\epsilon)\int_{\Omega}|\nabla u^{\epsilon}|^2\qquad \forall \epsilon\in (0,\epsilon_0).
\ee
We provide a proof of this inequality in the Appendix.

\medskip
Since $\ue$ is bounded in $H^1(\Om)$, we immediately infer that $\bar u_{|x_3=0}$ vanishes in $L^2(\T^2)$. Thus $(\bar u, \bar p)$ is a solution of the Stokes system with homogeneous Dirichlet boundary conditions at $x_3=0$ and $x_3=1$, i.e. $\bar u = \bu_\infty$.

\medskip

\textbf{Critical and sub-critical cases: $\aeps\lesssim\eps^2$.}
We  follow here the strategy of articles \refcite{Allaire1,Allaire2} by Allaire. These articles deal with the homogenization of the Stokes equations across a network of balls, with a Dirichlet condition  at the surface of the balls. Notably, in section 4 of article \refcite{Allaire1}, the  balls are assumed to be  distributed along a hypersurface (for instance, 3d balls with  centers  periodically located on a plane). In the setting considered here,  the rough idea is to extend the Stokes solution to the lower half-space by appropriate symmetry:  our problem  is then reduced to  the homogenization of the Stokes equations across a planar network of patches. Hence,  the ideas of Ref. \refcite{Allaire1}, devoted to a planar network of balls, essentially apply.  They are based on the construction of correctors and the method of oscillating test functions.  We start with
\begin{lemma}[Existence of correctors]\label{lemma:patches}
Assume that $\aeps\lesssim\eps^2$.
For every $\eps>0$, there exist $W^\eps=(W_{i,j}^\eps)_{1\leq i,j\leq 3}\in H^1(\Om)^9$, $q^\eps=(q_j^\eps)_{1\leq j\leq 3}\in L^2(\Om)^3$, supported in $\T^2\times [-\eps/2,\eps/2]$, which satisfy the following properties:
\begin{itemize}
\item[(i)] $W^\eps \rau 0$ weakly in $H^1(\Om)$, $q^\eps\rau 0$ weakly in $L^2(\Om)$;
\item[(ii)] for every $j=1\ldots 3$, $\sum_i\d_iW_{ij}^\eps = 0$ in $\Om$;
\item[(iii)] for $1\leq i,j\leq 3$, $W_{3j}^\eps=W_{i3}^\eps=0$ on $\T^2\times \{0\}$, and for $1\leq i,j\leq 2$, $W_{ij}^\eps=\delta_{ij}$ on ${\cal T}^\eps \times \{0\}$;
\item[(iv)] For every $\phi\in C^{\infty}(\overline{\Om})^3$, every $\psi\in H^1(\Om)^3$ and every sequence $\psi^\eps\in H^1(\Om)^3$ satisfying the boundary conditions
\be \label{BC:Psi-eps}
\psi^\eps = 0\ \textrm{on } ({\cal T}^\eps\times\{0\}) \cup (\T^2\times\{1\}),\quad \psi_3^\eps=0 \ \textrm{on }\T^2\times\{0,1\},
\ee
and converging weakly to $\psi$ in $H^1(\Om)^3$, the following relation holds: if $\lim_{\eps\to 0}\aeps/\eps^2=C_0\geq 0$, then 
\be \label{caracterisation:matrice-M}
\lim_{\eps\ra 0} \sum_{1\leq i,j\leq 3} \left(\int_{\Om} \nabla W_{ij}^\eps \cdot \nabla \psi_i^\eps \phi_j - \int_{\Om} \d_i \psi_i^\eps q_j^\eps \phi_j \right) = -C_0 \int_{\T^2\times \{0\}} M_0\psi_h\cdot \phi_h.
\ee
where   $M_0\in \Mc_2(\R)$ is the symmetric definite positive matrix  given by formula \eqref{M0}.

\end{itemize}

\end{lemma}

\begin{proof}[Proof of Lemma \ref{lemma:patches}]

This lemma is the analogue of Proposition 4.1.6 in Ref. \refcite{Allaire2} (see also section 2.3 in Ref. \refcite{Allaire1}). 
As mentioned before, we do not claim any major novelty in the proof. Nevertheless, with regards to  quantitative aspects, notably the exact expression of the slip matrix $C_0 \, M_0$, we feel necessary to reproduce its main steps.  

\medskip
The starting idea is to consider a base flow $(W,q)$ in the vicinity of $T$, which, after proper rescaling,  will  describe accurately the corrector  behavior near a single  patch. We shall  then truncate it  and periodize so as to obtain an appropriate global corrector.  
%Notice that the most standard way to build $W^\eps$ would be to consider a cell problem on the torus $\T^2$, namely
%to build a couple $(\we,Q^\eps)\in H^1(\T^2\times \R_+)^9\times H^{-1}(\T^2\times \R_+)^3$ such that
%$$
%\begin{aligned}
%-\Delta \we + \na Q^\eps=0\text{ in }\T^2\times \R_+,\\
%\dv \we =0,\\
%\we e_3=0\text{ on }\T^2\times\{0\},\\
%\pa_3 W^\eps_h=0\text{ on }(\aeps\eps^{-1} T)^c\times\{0\},\quad \we e_i=e_i\text{ on }(\aeps\eps^{-1} T)\times\{0\}\text{ for }i=1,2.
%\end{aligned}
%$$
%However, since $\aeps\ll \eps$, it is equivalent to build a solution of the above system in $\T^2\times \R_+$ or in $\R^2\times \R_+$. Thus we choose in the sequel the second alternative, as in the papers of Allaire. Nevertheless, we emphasize that in the numerical simulations presented in section \ref{sec:simul}, we will rather work with the cell problem above.
Namely, we introduce the solution $(W,q)$, with $\displaystyle W=(W_{ij})_{1\leq i,j \leq 3}$, $\displaystyle q=(q_i)_{1\leq i\leq 3}$, of the following problem:
\begin{align}
-\Delta W_{ij} + \d_i q_j = 0\quad & \textrm{in }\R^3_+,\ 1\leq i,j\leq 3,\label{eq:Stokes-demi-espace}\\
\sum_{i=1}^3\d_i W_{ij} = 0 \quad & \textrm{in }\R^3_+,\ 1\leq j\leq 3\label{eq:divergence-demi-espace},
\end{align}
completed with the boundary conditions
\begin{align}
W_{ij}=\delta_{ij} \quad & \textrm{on }T\times\{0\},\ 1\leq i,j \leq 2,\label{BC:deltaij-sur-T}\\
W_{i3}=W_{3j}=0\quad & \textrm{on }\R^2\times\{0\},\ 1\leq i,j \leq 3,\\
\d_3 W_{ij} = 0 \quad & \textrm{on }(\R^2\setminus T)\times\{0\},\ 1\leq i,j \leq 2.\label{BC:slip-sur-Tc}
\end{align}
as well as $\lim_{|x| \rightarrow \infty} W = 0$. Of course,  for $j=3$, we have  $W_{i3}\equiv 0$ and $q_3\equiv 0$. For $j = 1,2$, the existence of a unique weak solution  
$$(W_{\cdot j}, q_j)\in \left(D^{1,2}(\R^3_+)\right)^3 \times L^{2}_{loc}(\R_+^3)/\R$$
follows from Lax-Milgram theorem. Following Ref. \refcite{Galdi}, we remind that $D^{1,2}(\R^3_+)$ is the closure of ${\cal D}(\overline{\R_+^3})$ in $\dot{H}^1(\R_+^3)$.  

%\medskip
%{\em Estimate on the pressure.}
%Let us prove that, for every $j=1\ldots 3$, up to an additive constant, $q_j\in L^2(\R^3_+)$. Fix $j$ and a real $R>0$. Let $B_R$ be the open ball of radius $R$ in $\R^3$, centered at the origin, and define $B_R^+=B_R\cap \R^3_+$. Let $f\in L^2(B_R^+)$ such that $\int_{B_R^+}f=0$. By a classical argument (see for instance \cite{Galdi1994}), there exists a vector field $\phi\in H^1_0(B_R^+)$ and a constant $C_R>0$, depending only on $R$, such that $\dv \phi = f$ and 
%\be \label{estimate:nabla-phi}
%\|\nabla \phi\|_{L^2(B_R^+)}\leq C_R \|f\|_{L^2(B_R^+)}.
%\ee
%Using a scaling argument, we observe that $C_R$ is, in fact, independant of $R$.
%Multiplying Equation \eqref{eq:Stokes-demi-espace} by $\phi$ and integrating by parts, we obtain
%\be\label{formulation-faible:W-phi} 
%\int_{\R^3_+} \sum_{i=1}^3 \nabla W_{ij}\cdot \nabla \phi_i + \int_{\R^3_+}  q_j \sum_{i=1}^3 \d_i\phi_i = 0,\quad 1\leq j \leq 2. 
%\ee
%Combining relation \eqref{formulation-faible:W-phi} with estimate \eqref{estimate:nabla-phi}, we deduce the existence of a constant $C>0$, depending only on $\|\nabla W\|_{L^2(\R^3_+)}$, such that
%\bestar
%\left| \int_{B_R^+}q_j f \right| \leq C \|f\|_{L^2(B_R^+)},
%\eestar 
%and consequently,
%\bestar
%\|q_j\|_{L^2(B_R^+)/\R}\leq C.
%\eestar
%Letting $R\ra \infty$, we deduce that $q_j\in L^2(\R^3_+)/\R$, which yields the desired result.

\medskip
{\em Asymptotic behaviour of $W_{ij}, q_j$.} For $j=1\ldots 3$, we extend $q_j$, $W_{1j}$ and $W_{2j}$ into even functions of $x_3$, and  $W_{3j}$ into an odd function of $x_3$. We obtain in this way solutions of the Stokes equations on  $\R^3\setminus T$. Proceeding exactly as in page 255 of Ref. \refcite{Allaire1}, we obtain the following asymptotic expansions
\begin{align}
%W_{\cdot j}(x) \: = \: \frac{1}{8\pi}\, \frac{F_j}{|x|}+ O(\frac{1}{|x|^2})\quad \textrm{as }|x|\ra\infty.\label{DA:Wij}
%\\
W_{\cdot j}(x) \:  & = \: \frac{1}{8\pi}\, \left(\frac{F_j}{|x|} + \frac{(x\cdot F_j)x}{|x|^3}\right)+ O\left(\frac{1}{|x|^2}\right)\quad \textrm{as }|x|\ra\infty.\label{DA:Wij}
\\
q_j(x) \: & = \: \, \frac{1}{4\pi} \frac{x\cdot F_j}{|x|^3}+ O\left(\frac{1}{|x|^3}\right)\quad \textrm{as }|x|\ra\infty,\label{DA:qj}
\end{align}
In formulas \eqref{DA:Wij}-\eqref{DA:qj}, the notation $F_j$ corresponds to the drag force, which is defined by (here, $n_+ \: := \: e_3$, $n_- \: := \: -e_3$): 
\begin{equation} \label{drag}
\begin{aligned}
F_j \:  & =  \: - \int_{T \times \{0^+\}}\frac{\d W_{\cdot j}}{\d n_+} - \int_{T \times \{0^-\}}\frac{\d W_{\cdot j}}{\d n_-} + \int_{T \times \{0^+\}} q_j n_+ + \int_{T \times \{0^-\}} q_j n_- \:  \\ 
& = \:  - 2\int_{T \times \{0^+\}} \d_3 W_{\cdot j}.
\end{aligned}
\end{equation}

\medskip
{\em Construction of $W^\eps$ and $q^\eps$.}  Using the extended $W$ and $q$, defined in the whole of $\R^3$,  we can then proceed exactly as in Ref. \refcite{Allaire1,Allaire2} to construct the correctors $W^\eps$ and $q^\eps$. Therefore, we consider the following decomposition of $P_k^\eps$ (see definition \eqref{def:cube-Pk}):
\bestar
\overline{P_k^\eps}=C_k^\eps \cup \overline{D_k^\eps}\cup \overline{K_k^\eps},
\eestar
where $C_k^\eps$ is the ball of radius $\eps/4$ centered in the cube, $D_k^\eps$ is the ball of radius $\eps/2$, with same center, perforated by $C_k^\eps$, and $K_k^\eps$ is the remaining part of the cube, that is $K_k^\eps=P_k^\eps\setminus \overline{D_k^\eps}$ (see Figure \ref{fig:decomp}). We denote by $c_k^\eps$ the center of cube $P_k^\eps$. In each part of the cube, we define $W_{\cdot j}^\eps$ and $q_j^\eps$ as follows:
\bestar
\left\lbrace 
\begin{matrix}
W_{\cdot j}^\eps(x)=W_{\cdot j}(\frac{x-c_k^\eps}{\aeps})\\
q_j^\eps(x) = \frac{1}{\aeps} q_j(\frac{x-c_k^\eps}{\aeps})
\end{matrix}  \quad \forall x\in C_k^\eps, \right. ,
\left\lbrace
\begin{matrix}
\nabla q_j^\eps - \Delta W_{\cdot j}^\eps = 0\\
\dv W_{\cdot j}^\eps = 0
\end{matrix}
\right.
\quad \textrm{in }D_k^\eps,
\qquad
\left\lbrace 
\begin{matrix}
W_{\cdot j}^\eps = 0\\
q_j^\eps = 0
\end{matrix}
\right. \quad \textrm{in }K_k^\eps.
\eestar
Moreover, we impose $\int_{D_k^\eps}q_j^\eps = 0$ and $W_{\cdot j}^\eps\in H^1(P^\eps_k)^3$ (so that there is no jump of $W^\eps$ across $\pa D^\eps_k$, $\pa C^\eps_k$). 

\begin{figure}
\begin{center}
\includegraphics[scale=0.30]{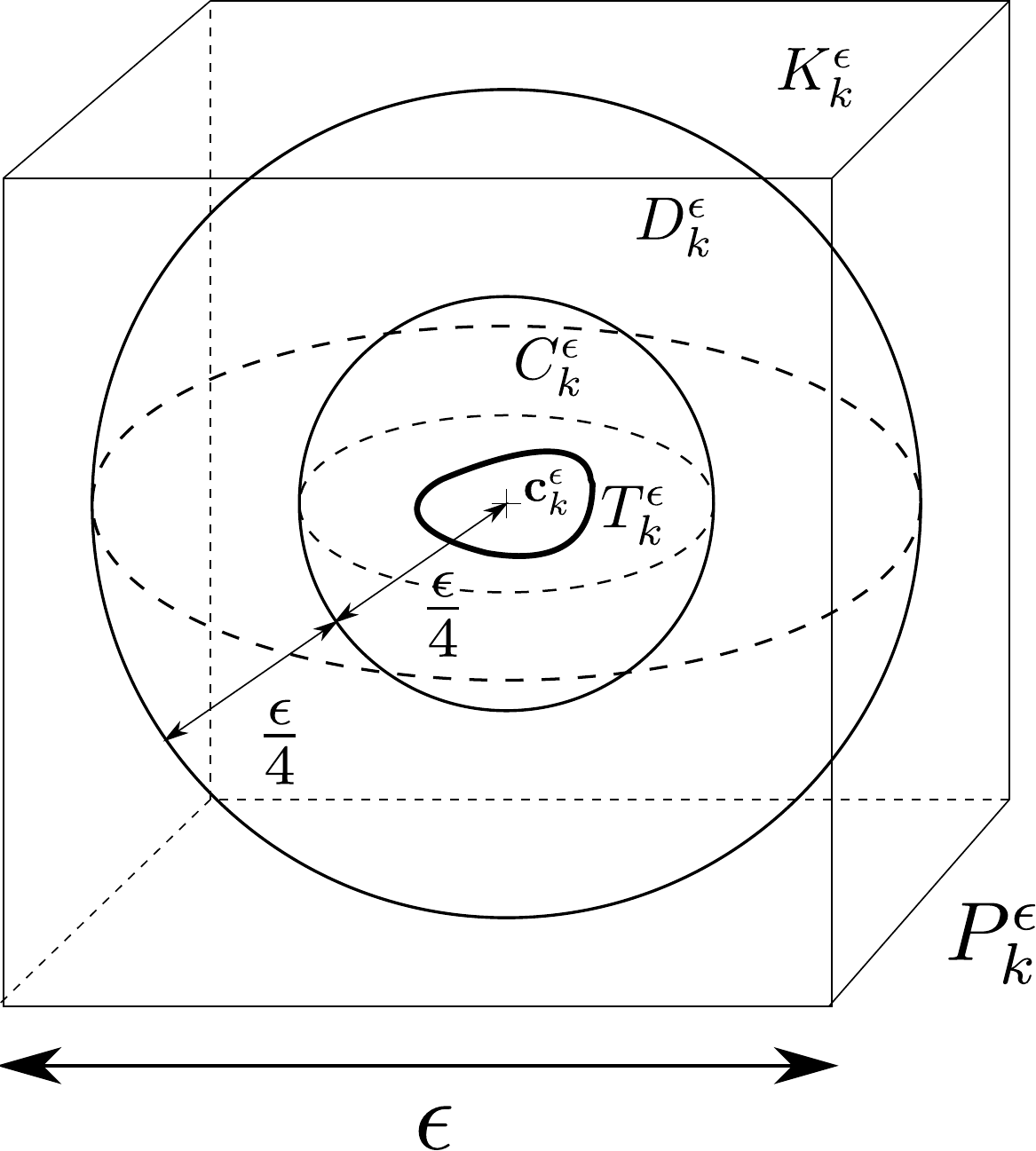}
\caption{Each cube $P_k^\eps$  is decomposed into a union of subdomains $C_k^\eps$, $D_k^\eps$ and $K_k^\eps$, which are separated by spheres of radius $\eps/4$ and $\eps/2$ centered in the cube.}
\label{fig:decomp}
\end{center}
\end{figure}

\medskip 
{\em Estimates on $W^\eps$ and $q^\eps$.} 
%To estimate $L^2$ norms in $\Om$, we will derive estimates in $P_k^\eps$ and sum them over $k$. 
We use  again the decomposition $\displaystyle P_k^\eps=C_k^\eps \cup (P_k^\eps \setminus \overline{C_k^\eps})$.
The estimates in $C_k^\eps$ follow from  the asymptotic expansions \eqref{DA:Wij}-\eqref{DA:qj} and a scaling argument:  for every $\eps>0$,
$$
\|\nabla W_{\cdot j}^\eps\|_{L^2(C_k^\eps)}^2 \leq C \aeps, \ 
\| q_j^\eps\|_{L^2(C_k^\eps)}^2 \leq C \aeps, \ 
\| W_{\cdot j}^\eps\|_{L^2(C_k^\eps)}^2  \leq C \aeps^2 \eps,
$$
where $C>0$ is a constant. To treat the remaining part $P_k^\eps \setminus \overline{C_k^\eps}$, we use a properly rescaled version of standard estimates for the homogeneous Stokes equations: basically, the $L^2$ norm, resp.\,$H^1$ norm of the solution is controlled by the $L^2$ norm, resp.\,$H^{1/2}$ norm of the boundary data (see for instance Ref. \refcite{Solonnikov-Scadilov}). Since the velocity fields $W_{\cdot j}^\eps$ satisfy the following pointwise asymptotics as $\eps$ vanishes
$$
W_{\cdot j}^\eps  = O\left(\frac{\aeps}{\eps}\right)\quad \textrm{on }\d C_k^\eps \cap \d D_k^\eps,\  
\quad \nabla W_{\cdot j}^\eps  = O\left(\frac{\aeps}{\eps^2}\right)\quad \textrm{on }\d C_k^\eps \cap \d D_k^\eps,
$$
using a scaling argument,  we obtain the following estimates
\be\label{estimate:P-C}
\|\nabla W_{\cdot j}^\eps\|_{L^2(P_k^\eps \setminus \overline{C_k^\eps})}^2 \leq C \frac{\aeps^2}{\eps}, \ 
\| q_j^\eps\|^2_{L^2(P_k^\eps \setminus \overline{C_k^\eps})}  \leq C \frac{\aeps^2}{\eps}, \ 
\| W_{\cdot j}^\eps\|_{L^2(P_k^\eps \setminus \overline{C_k^\eps})}^2 \leq C \aeps^2 \eps,
\ee
for a given  constant $C>0$.
Since $0<\aeps<\eps$, we deduce 
$$
\|\nabla W_{\cdot j}^\eps\|_{L^2(P_k^\eps)}^2  \leq C \aeps, \ 
\|q_j^\eps\|_{L^2(P_k^\eps)}^2  \leq C \aeps, \ 
\| W_{\cdot j}^\eps\|_{L^2(P_k^\eps)}^2  \leq C \aeps^2 \eps.
$$
As a result, summing over $k\in [|0,\eps^{-1}|]^2$, we obtain the following asymptotics as $\eps$ vanishes
\be \label{asymptotics:Weps-Qeps}
\|\nabla W^\eps\|_{L^2(\Om)}^2=O\left(\frac{\aeps}{\eps^2}\right),\quad \|q^\eps\|^2_{L^2(\Om)}=O\left(\frac{\aeps}{\eps^2}\right), \quad 
\|W^\eps\|_{L^2(\Om)}^2=O\left(\frac{\aeps^2}{\eps}\right).
\ee

\medskip

{\em Conclusion of the proof. }
Let $\phi\in C^{\infty}(\overline{\Om})^3$, $\psi\in H^1(\Om)^3$ and let $\psi^\eps\in H^1(\Om)^3$ be a sequence of vector fields satisfying the boundary conditions  \eqref{BC:Psi-eps}, and converging weakly to $\psi$ in $H^1(\Om)^3$. In the sub-critical case $\aeps\ll \eps^2$, the asymptotics \eqref{asymptotics:Weps-Qeps} imply that
\bestar
W^\eps \ra 0\quad \textrm{strongly in }H^1(\Om)^9, \quad q^\eps \ra 0\quad \textrm{strongly in }L^2(\Om).
\eestar
Consequently, the following relation holds:
\bestar
\lim_{\eps\ra 0} \sum_{1\leq i,j\leq 3} \left(\int_{\Om} \nabla W_{ij}^\eps \cdot \nabla \psi_i^\eps \phi_j - \int_{\Om} \d_i \psi_i^\eps q_j^\eps \phi_j \right) =  0.
\eestar
Thus, relation \eqref{caracterisation:matrice-M} holds with $\overline{M}=0$. 

\medskip
In the critical case $\lim_{\eps\ra 0} \frac{\aeps}{\eps^2}=C_0>0$, we define $\tilde \Om= \T^2\times (-1,1)$, and we
extend $\phi_j$ and $\psi_j^\eps$ into even functions of $x_3$ on $\tilde \Om$ for $j=1,2$, and $\phi_3$ and $\psi_3$ into odd functions of $x_3$. 
First, asymptotics \eqref{asymptotics:Weps-Qeps} imply that $W^\eps$ is bounded in $H^1$, and therefore converges weakly in $H^1$, up to a subsequence. Since $W^\eps$ vanishes in $L^2(\Om)$, we obtain $\nabla W^\eps\rau 0$ weakly in $L^2(\Om)^9$.
From \eqref{asymptotics:Weps-Qeps}, we also infer $q_j^\eps\rau 0$ weakly in $L^2(\Om)$, and thus the following identity holds for every $1\leq i \leq 3$, $1\leq j\leq 2$
\begin{align*}
 \int_{\Om} \nabla W_{ij}^\eps \cdot \nabla \psi_i^\eps \phi_j - q_j^\eps (\d_i\psi_i^\eps)  \phi_j
& = \frac{1}{2} \int_{\tOm} \nabla W_{ij}^\eps \cdot \nabla \psi_i^\eps \phi_j - q_j^\eps (\d_i\psi_i^\eps)  \phi_j
\\
& =\frac{1}{2}  \int_{\tOm} \nabla W_{ij}^\eps \cdot \nabla (\psi_i^\eps \phi_j) - q_j^\eps \d_i(\psi_i^\eps \phi_j)
 + o(1),\quad \textrm{as }\eps\ra 0.
\end{align*}
Moreover,
\begin{align*}
& \int_{\tOm} \nabla W_{ij}^\eps \cdot \nabla (\phi_j \psi_i^\eps) -  q_j^\eps \d_i(\phi_j \psi_i^\eps ) 
\/ = \: \sum_{k} \int_{P_k^\eps} \nabla W_{ij}^\eps \cdot \nabla (\phi_j \psi_i^\eps) -  q_j^\eps \d_i(\phi_j \psi_i^\eps)  \\
= & \sum_{k} \int_{C_k^{\eps,+}} \nabla W_{ij}^\eps \cdot \nabla (\phi_j \psi_i^\eps) - q_j^\eps \d_i(\phi_j \psi_i^\eps)  \\
+ & \sum_{k}  \int_{C_k^{\eps,-}} \nabla W_{ij}^\eps \cdot \nabla (\phi_j \psi_i^\eps) - q_j^\eps \d_i(\phi_j \psi_i^\eps) 
+\sum_{k}  \int_{P_k^\eps \setminus C_k^\eps} \nabla W_{ij}^\eps \cdot \nabla (\phi_j \psi_i^\eps) - q_j^\eps \d_i(\phi_j \psi_i^\eps) ,
\end{align*}
where $C_k^{\eps,\pm}= C_k^\eps \cap \R^3_\pm$. In all sums, $k$ ranges over $[|0,\eps^{-1}|]^2$. Using the estimates \eqref{estimate:P-C}, we infer that
\begin{eqnarray*}
\sum_{k}  \int_{P_k^\eps \setminus C_k^\eps} \nabla W_{ij}^\eps \cdot \nabla (\phi_j \psi_i^\eps) - q_j^\eps \d_i(\phi_j \psi_i^\eps) &\leq& C\left( \| \na W^\eps\|_{L^2(\cup_k P_k^\eps \setminus C_k^\eps)}+ \|q^\eps\|_{L^2(\cup_k P_k^\eps \setminus C_k^\eps)}\right)\\
&\leq& C \left(\frac{\aeps^2}{\eps} \frac{1}{\eps^2}\right)^{1/2}\ll 1.
\end{eqnarray*}
At this stage the proof differs slightly from the one of Ref. \refcite{Allaire2}, because of the mixed boundary conditions at $x_3=0$. Indeed, since $(W^\eps, q^\eps)$ satisfies the Stokes system in $C^{\eps,\pm}_k$, we have
$$
\int_{C_k^{\eps,\pm}} \nabla W_{ij}^\eps \cdot \nabla (\phi_j \psi_i^\eps) - \int_{C_k^{\eps,\pm}}q_j^\eps \d_i(\phi_j \psi_i^\eps) = \int_{\d C_k^{\eps,\pm}}
\left( \frac{\d W_{ij}^\eps}{\d n} - q_j^\eps n\cdot e_i \right) \phi_j \psi_i^\eps,
$$
where $n$ denotes the outer normal to the set $C_k^{\eps,\pm}$. In particular, due to the symmetry properties of $W^\eps,q^\eps, \phi, \psi^\eps$, there holds
\begin{align*}
&\int_{C_k^{\eps,+}} \nabla W_{ij}^\eps \cdot \nabla (\phi_j \psi_i^\eps) - \int_{C_k^{\eps,+}}q_j^\eps \d_i(\phi_j \psi_i^\eps) 
+ \int_{C_k^{\eps,-}} \nabla W_{ij}^\eps \cdot \nabla (\phi_j \psi_i^\eps) - \int_{C_k^{\eps,-}}q_j^\eps \d_i(\phi_j \psi_i^\eps) \\
=&\int_{\pa C_k^{\eps}}
\left( \frac{\d W_{ij}^\eps}{\d n} - q_j^\eps n\cdot e_i \right) \phi_j \psi_i^\eps - 2 \int_{C_k^\eps\cap\{z=0\}} \pa_3 W_{ij}^\eps \phi_j \psi_i^\eps.
\end{align*}
By definition of $W$, $\: \pa_3 W_{ij}^\eps =0$ on $(C_k^\eps\cap\{x_3=0\})\setminus( {\cal T}^\eps \times\{0\})$. On the other hand, since $\psi^\eps$ satisfies \eqref{BC:Psi-eps}, $\psi_i^\eps=0$ on ${\cal T}^\eps \times\{0\}$. Therefore, the r.h.s.\,reduces to the integral on $\pa C^\eps_k$. From now on, in order to avoid confusion, we denote by $n_k$ the normal vector to  the ball $C^\eps_k$.
Using the asymptotic expansions \eqref{DA:Wij}-\eqref{DA:qj} and the expression of $W_{ij}^\eps, q_j^\eps$ in $C_k^\eps$, we obtain, on $\pa C_k^\eps$,
\begin{equation*}
\frac{\d W_{ij}^\eps}{\d n_k} - q_j^\eps n_k\cdot e_i \: = \: -\frac{\aeps}{\eps^2}\left[\frac{2}{\pi} F_{ij} +\frac{6}{\pi}   e_i\cdot n_k F_j\cdot n_k\right]\\
\: + \: \left(\frac{\aeps}{\eps^2}\right)^2 \eps R_{ij}^\eps,
\end{equation*}
where 
$R_{ij}^\eps$ is a function of $x$, satisfying $R_{ij}^\eps(x)=O(1)$ as $\eps\ra 0$, uniformly in $x$ and $k$. This leads to the following decomposition
\begin{align*}
\int_{\tOm} \nabla W_{ij}^\eps \cdot \nabla (\phi_j \psi_i^\eps) - \int_{\tOm} q_j^\eps \d_i(\phi_j \psi_i^\eps)
& = -  \frac{\aeps}{\eps^2} \sum_{k}  \int_{\d C_k^\eps}
\left[ \frac{2}{\pi} F_{ij} + \frac{6}{\pi} e_i\cdot n_k F_j\cdot n_k\right] \phi_j \psi_i^\eps \\&+   \left( \frac{\aeps}{\eps^2} \right)^2 \sum_{k}  \int_{\d C_k^\eps} \eps R_{ij}^\eps
\phi_j \psi_i^\eps+o(1).
\end{align*}
Let $\delta_{\d C_k^\eps}$ be the unit mass concentrated on $\d C_k^\eps$. We  use the following Lemma, proved by Allaire\footnote{Notice that in the paper of Allaire, the periodicity of the pattern is $2\eps$, rather than $\eps$ as in the present paper. Hence the constant in front of the Dirac mass in the right-hand side is $\pi/16$, rather than  $\pi/64$ for the first line, and $\pi/48$ rather than $\pi/192$ in the second line.}:
\begin{lemma}[see Lemma 4.2.1 in Ref \refcite{Allaire2}]
\be \label{convergence:Dirac-mass}
\begin{aligned}
\sum_{k} \delta_{\d C_k^\eps} \ra \frac{\pi}{16}\delta_{\T^2\times \{0\}}\quad \textrm{strongly in }H^{-1}(\tilde{\Om}),\\
\sum_{k} e_i \cdot  n_k n_k\delta_{\d C_k^\eps} \ra \frac{\pi}{48}e_i\delta_{\T^2\times \{0\}}\quad \textrm{strongly in }H^{-1}(\tilde{\Om}).
\end{aligned}
\ee
\end{lemma}

Let us write
\be \label{identity:Dirac-decomposition}
\begin{aligned}
\sum_{k}  \int_{\d C_k^\eps}
\phi_j \psi_i^\eps=\left\langle \sum_{k\in \Kc^\eps} \delta_{\d C_k^\eps}, \phi_j \psi_i^\eps \right\rangle_{H^{-1}(\tilde{\Om})\times H^1_0(\tilde{\Om})},\\
\sum_{k}  \int_{\d C_k^\eps}
e_i \cdot  n_k n_k\phi_j \psi_i^\eps=\left\langle \sum_{k} e_i \cdot  n_k n_k\delta_{\d C_k^\eps}, \phi_j \psi_i^\eps \right\rangle_{H^{-1}(\tilde{\Om})\times H^1_0(\tilde{\Om})};
\end{aligned}
\ee 
consequently, since $\phi_j \psi_i^\eps\ra \phi_j \psi_i$ weakly in $H^1(\tilde{\Om})$ and $\frac{\aeps}{\eps^2}\ra C_0$, we obtain 
\be \label{limite:1st-Dirac}
\lim_{\eps\ra 0} \frac{\aeps}{\eps^2} \sum_{k}  \int_{\d C_k^\eps}
\left[ \frac{2}{\pi} F_{ij} + \frac{6}{\pi} F_j\cdot n_k n_k\cdot e_i\right] \phi_j \psi_i^\eps= \frac{1}{4}C_0\int_{\T^2\times\{0\}} F_{ij} \phi_j \psi_i.
\ee
Moreover, since $R_{ij}^\eps$ is uniformly bounded in $L^\infty(\tilde{\Om})$, we use  the following comparison principle:
\begin{lemma}[see Lemma 2.3.8 in Ref. \refcite{Allaire1}]
Let $\alpha_\eps$ and $\beta_\eps$ be two positive functions in $H^{-1}(\Om)$ such that
$$0\leq \alpha_\eps\leq \beta_\eps.$$
If $\beta_\eps$ converges strongly to zero in $H^{-1}(\Om)$, then so does $\alpha_\eps$.
\end{lemma}
Whence we deduce from \eqref{convergence:Dirac-mass} that
\bestar
\eps \sum_{k} R_{ij}^\eps \delta_{\d C_k^\eps} \ra  0\quad \textrm{strongly in }H^{-1}(\tilde{\Om}).
\eestar
Using \eqref{identity:Dirac-decomposition}, we obtain the following convergence:
\be \label{limite:2nd-Dirac}
\left(\frac{\aeps}{\eps^2}\right)\sum_{k}  \int_{\d C_k^\eps} \eps R_{ij}^\eps \phi_j \psi_i^\eps\ra 0\quad\textrm{as }\eps \ra 0.
\ee
Gathering the convergence results \eqref{limite:1st-Dirac}-\eqref{limite:2nd-Dirac}, we obtain relation \eqref{caracterisation:matrice-M}, where the matrix $M_0$ is defined by 
\begin{equation} \label{M0}
M_{0,ij} \: = \:  \frac{1}{8} F_{ij}, \quad F_{\cdot j}  \:\: \mbox{ given by  \eqref{drag}} 
\end{equation}

\medskip
There only remains to prove that the matrix $(F_{ij})_{1\leq i,j\leq 2}$ is negative definite. To that end, we   go back to system \eqref{eq:Stokes-demi-espace}-\eqref{BC:slip-sur-Tc}. We multiply  by $W_{\cdot i}$ the system satisfied by $W_{\cdot j}$, and we obtain
$$
F_{ij}= -2 \int_{T \times \{0\}} \pa_3 W_{ij}= -2 \int_{\R^3_+}\na W_{\cdot i} : \na W_{\cdot j}.
$$
In particular, for all $\eta \in \R^2$,
$$
\sum_{1\leq i,j\leq 2} \eta_i \eta_j F_{ij}= -2 \int_{\R^3_+} \left| \na (\eta_1 W_{\cdot 1} + \eta_2 W_{\cdot 2})\right|^2\leq 0,
$$
and the right-hand side above vanishes if and only if $\eta_1 W_{\cdot 1} + \eta_2 W_{\cdot 2}=0$ a.e. in $\R^3_+$. In view of the boundary conditions \eqref{BC:deltaij-sur-T}, this implies $\eta_1=\eta_2=0$.
 This concludes the proof of Lemma \ref{lemma:patches}.

\end{proof}
\medskip

To complete the proof of Theorem \ref{thm:patch}, we rely on Lemma \ref{lemma:patches}, as follows.
Let $\phi\in C^{\infty}(\overline{\Om}^3)$ satisfying the no-slip condition $\phi=0$ on the upper boundary $\T^2\times\{1\}$, and the non-penetration condition $\phi_3=0$ on the lower boundary $\T^2\times\{0\}$. Let $W^\eps\in H^1(\Om)^9$, $q^\eps\in L^2(\Om)^3$ be the sequences introduced in Lemma \ref{lemma:patches}. We define the following test functions for the weak formulation associated to system \eqref{eq:Stokes}-\eqref{mixed-z=0}:
\bestar
\phi^\eps = (\mathrm{I}_3 - W^\eps)\phi,\quad r^\eps = q^\eps \phi,
\eestar
where $\mathrm{I}_3$ is the identity matrix in $\Mc_3(\R)$.
We deduce the following relation:
\begin{align}
\int_{\Om} \nabla u^\eps : \nabla \phi^\eps - \int_{\Om} p^\eps \dv \phi^\eps  & = \int_{\Om} f \phi^\eps \label{weak-formulation-velocity}\\
\int_{\Om} r^\eps \dv u^\eps & =0. \label{weak-formulation-pressure}
\end{align}
Since $W^\eps$ converges weakly to $0$ in $H^1(\Om)^9$, and strongly to $0$ in $L^2(\Om)^9$, we readily obtain
\begin{align*}
\int_{\Om} \nabla u^\eps : \nabla \phi^\eps & = \int_{\Om} \nabla u^\eps : \nabla \phi - \int_{\Om} \nabla u_i^\eps \nabla W_{ij}^\eps \phi_j - \int_{\Om} \nabla u_i^\eps \nabla  \phi_j W_{ij}^\eps\\
& = \int_{\Om} \nabla \overline{u} : \nabla \phi - \int_{\Om} \nabla u_i^\eps \nabla W_{ij}^\eps \phi_j+o(1),\quad \textrm{as }\eps\ra 0,\\
-\int_{\Om}p^\eps \dv \phi^\eps & = - \int_{\Om}p ^\eps (\mathrm{I}_3-W^\eps):\nabla \phi \: = \: -\int_{\Om}\overline{p}\ \dv \phi + o(1),\quad \textrm{as }\eps\ra 0.
\end{align*}
Consequently, summing relations \eqref{weak-formulation-velocity} and \eqref{weak-formulation-pressure}, we deduce the  asymptotic relation
\bestar
\int_{\Om} \nabla \overline{u}:\nabla\phi - \int_{\Om}\overline{p}\ \dv \phi - \int_{\Om} \nabla u_i^\eps \nabla W_{ij}^\eps \phi_j + \int_{\Om} q^\eps \phi\ \dv u^\eps = o(1),\quad \textrm{as }\eps\ra 0.
\eestar
Applying Lemma \ref{lemma:patches} with $\psi=\overline{u}$ and $\psi^\eps=u^\eps$, we obtain the following relation:
\be\label{weak-formulation-ubarre}
\int_{\Om} \nabla \overline{u}:\nabla\phi - \int_{\Om}\overline{p}\ \dv \phi+C_0\int_{\T^2\times\{0\}}M_0 u \cdot \phi = 0,
\ee
where the matrix  $M_0\in\Mc_{2}(\R)$ is defined by \eqref{M0}.
Since relation \eqref{weak-formulation-ubarre} holds for every test function $\phi$, this proves that $\overline{u}=\overline{u}_{C_0M_0}$.

%%%%%%%%%%%%%%%%%%%%%%%%%%%%%%%%%%%%%%%%%%%%%%%%%%%%%%
%%%%% Remarque voulue par le referee apres la premiere revision %%%%%%
%%%%%%%%%%%%%%%%%%%%%%%%%%%%%%%%%%%%%%%%%%%%%%%%%%%%%%

\begin{remark}
Theorem 1 expresses that  the homogenized boundary condition depends strongly on the ratio between  slip and no-slip areas. By simple symmetry, the velocity can be extended though the  \textit{planar} slip zones into a Stokes solution satisfying Dirichlet conditions at the remaining part of the boundary.
In this way, the problem becomes very similar to the one raised by Allaire in Section 4 of  Ref. \refcite{Allaire2} on fluid flows through porous grids. In this respect,  it is different  from article  \refcite{AllaireSlip91} where Allaire considers slip conditions on \textit{volumic} obstacles (for which an extension like the one mentioned above cannot be performed).
%
%This is specific to the planar configuration considered here : in the case of volumic obstacles, the results by Allaire \cite{Allaire1,Allaire2} (extended by the same author to the case of slip boundary conditions ) highlight that slip and no-slip boundary conditions lead to the same critical scalings, which, roughly speaking, signifies that both types of conditions produce a similar effect on the flow.

\end{remark}

%%%%%%%%%%%%%%%%%%%%%%%%%%%%%%%%%
%%%%%%%%% Section 4 : riblets %%%%%%%%%%
%%%%%%%%%%%%%%%%%%%%%%%%%%%%%%%%

\section{Asymptotic study of ``riblet'' designs}
\label{sec:riblet}

This section is devoted to the proof of Theorem \ref{thm:riblet}. 
%The idea is that, by the  invariance of  system \eqref{eq:Stokes}-\eqref{noslip-z=1}-\eqref{mixed-z=0} with respect to $x_1$, one can reduce the dimension of the problem. 
%Namely, one can reduce the three-dimensional problem to a combination of two decoupled problems, set in the bidimensional domain $\T\times (0,1)$, which involve respectively the velocity of the fluid in the direction of the riblets, and the orthogonal components of the velocity. Then we apply homogenization results for the Laplace operator (resp. for the Stokes operator), that are proved in \cite{terme-etrange} (resp. in \cite{Allaire1}-\cite{Allaire2})  to identify the matrix $M_{rib}$.
In the case of riblets,  we recall that ${\cal T}^\eps$ is invariant by translation in $x_1$. Since $f=(f_1,f_2,f_3)$ is also independent on the $x_1$ variable, the solution $(u^\eps,p^\eps)$ of system \eqref{eq:Stokes}-\eqref{noslip-z=1}-\eqref{mixed-z=0} depends only on $(x_2,x_3)$. As a result, the first component of $u^\eps$ satisfies:
\be \label{Laplace2d}
\begin{aligned}
-\Delta_{2,3} u_1^\eps & = f_1  \quad \textrm{in }\T \times(0,1),\\
u_1^\eps & = 0\quad \textrm{on }\T \times\{1\},\\
\d_3 u_1^\eps = 0\ \textrm{on }(\T \times\{0\})\setminus (\Pi {\cal T}^\eps),\quad u_1^\eps & =0 \ \textrm{on }\Pi  {\cal T}^\eps,
\end{aligned}
\ee
where $\nabla_{2,3}$ and $\Delta_{2,3}$ stand for the gradient (resp.\,the Laplacian) with respect to the $(x_2,x_3)$ variables, $\T^1=\R/\Z $ and where we have denoted $\Pi$ the projection operator defined by $\Pi(x_1,x_2,0)=(x_2,0)$.
In the same fashion, $(u^\eps_2,u^\eps_3), p^\eps$ satisfy the following Stokes problem:
\be \label{Stokes2d}
\begin{aligned}
-\Delta_{2,3} \left( \begin{matrix}
u_2^\eps\\ u_3^\eps
\end{matrix} \right)
+ \nabla_{2,3} p^\eps & = \left(
\begin{matrix}
f_2 \\ f_3
\end{matrix}
\right)  \quad \textrm{in }\T\times(0,1), \\
\nabla_{2,3}\cdot \left( \begin{matrix}
u_2^\eps\\ u_3^\eps
\end{matrix} \right)  & = 0 \quad \textrm{in }\T\times(0,1),\\
u_2^\eps=u_3^\eps & = 0\quad \textrm{on }\T\times\{1\},\\
u^\eps_3 &= 0\quad \textrm{on }\T\times\{0\},\\
\pa_{3} u_{2}^\eps=0\text{ on }(\T\times\{0\}) \setminus (\Pi T^\eps), \quad u^\eps_{2|x_3=0} &=0\quad \textrm{on } \Pi{\cal T}^\eps.
\end{aligned}
\ee
Hence, the original 3d problem reduces to the study of two independent systems (with a Laplace and a Stokes equations), set in the 2d domain $\T \times (0,1)$. This change from a 3d to a 2d setting explains the change of scalings between Theorem \ref{thm:patch} and Theorem \ref{thm:riblet}.

%Let $(\overline{u},p)$ be the weak limit of $(u^\eps,p^\eps)$ in $H^1(\Om)^3\times L^2(\Om)/\R$. We now investigate the limits of \eqref{Laplace2d}, %\eqref{Stokes2d} separately.

\medskip 
{\em To handle the Stokes equations \eqref{Stokes2d}}, we proceed like in the previous section: in short,  we adapt the homogenization techniques of Ref. \refcite{Allaire1,Allaire2}, dedicated to  the Stokes flow across a periodic network of balls,  set along an hypersurface. As mentioned before, the difference is the dimension of the domain. One must this time consider the 2d results of Ref. \refcite{Allaire2}, about periodic network of disks along a line.  For brevity, we do not give further details. We eventually obtain the following limit system:
\be \label{limit-Stokes2d}
\begin{aligned}
-\Delta_{2,3} \left( \begin{matrix}
\overline{u}_2\\ \overline{u}_3
\end{matrix} \right)
+ \nabla_{2,3} \overline{p}  = \left(
\begin{matrix}
f_2 \\ f_3
\end{matrix}
\right)  \quad \textrm{in }\T^1\times(0,1), \\
\nabla_{2,3}\cdot \left( \begin{matrix}
\overline{u}_2\\ \overline{u}_3
\end{matrix} \right)  = 0 \quad \textrm{in }\T^1\times(0,1),\\
\overline{u}_{2}=\overline{u}_{3} = 0\quad \textrm{on }\T^1\times\{1\},\\
\overline{u}_3=0\text{ on }\T^1\times\{0\},\\
\pa_{3} \overline{u}_{2}=\frac{2\pi}{C_0}\overline{u}_{2}\textrm{ on }\T^1\times\{0\},
\end{aligned}
\ee
where we recall that $C_0:=\lim_{\eps \to 0} - \eps \ln |\aeps|$.
%$\bullet$ First, we can use a combination of the methods of \cite{terme-etrange,Allaire2}, and of the present paper in order to determine the system satisfied by $\bu_1$. Indeed, in \cite{terme-etrange}, D. Cioranescu and F. Murat investigate the homogenization of the Laplace system with a \textit{volume distribution of holes}, with Dirichlet boundary conditions on the bottom and on the surface. In this case, they prove that the limit system is
%$$
%\begin{aligned}
%-\Delta_{2,3} v_1 + \frac{\pi}{2C_1}v_1=f_1\quad \text{in }\T^1\times (0,1),\\
%v_{1|z=0}=v_{1|z=1}=0,
%\end{aligned}
%$$
%where $C_1=\lim_{\eps\to 0}-\eps \ln a_{2\eps}$ (notice that the periodicity in \cite{terme-etrange} is $2\eps$, as in \cite{Allaire1,Allaire2}), where $\aeps$ still denotes the size of a hole in a square of size $\eps$.
\medskip
{\em As regards the Laplace equation \eqref{Laplace2d}}, the idea is exactly the same. Actually, the situation is even simpler, and has been analysed for a longer time. Namely, one may start from the work of Cioranescu and  Murat (see Ref. \refcite{terme-etrange}), instead of section 4 in Ref. \refcite{Allaire1}. Again, we leave the details to the reader. In our setting, the limit system is  
\be \label{limit-Laplace2d}
\begin{aligned}
-\Delta_{2,3} \bar{u}_1 = f_1  \quad \textrm{in }\T^1\times(0,1),\\
\bar{u}_1 = 0\quad \textrm{on }\T^1\times\{1\},\\
\d_3 \bar{u}_1 = \frac{\pi}{C_0}  \bar{u}_1 \ \textrm{on }\T^1\times\{0\}.
\end{aligned}
\ee

\medskip
We deduce from systems \eqref{limit-Laplace2d} and \eqref{limit-Stokes2d} that $\overline{u}=\overline{u}_{M_{riblets}}$, $M_{riblets}$ being given by \eqref{Mparallel}. The sub-cases where $f=e_1$ or $f=e_2$ follow easily.

%%%%%%%%%%%%%%%%%%%%%%%%%%%%%%%
%%%%%% Section 5 : numerique %%%%%%%%
%%%%%%%%%%%%%%%%%%%%%%%%%%%%%%

\section{Numerical simulations}
\label{sec:simul}
This section is devoted to simulations of system  \eqref{eq:Stokes}-\eqref{noslip-z=1}-\eqref{mixed-z=0}. 
For simplicity, we shall restrict to constant source term (average pressure gradient), say 
$$ f \: = \: 2 e, \quad e \in \mbox{span}(e_1,e_2). $$
The idea is to recover numerically the scalings for the slip length given in Theorems 1 and 2. 
However, to observe significant slip implies to consider very small scales: patches of size less than $\eps^2$, in a grid of side $\eps$. This forbids direct computations. To overcome this difficulty, we shall rely on a boundary layer approximation of the Stokes flow. Such approximation, often implicitly used in physics papers, has been fully justified in the context of wall laws: see References \refcite{JaMi,DaGe,AmBrLe} among many others. 
 
\medskip
The starting point is to write the exact solution $u^\eps$ as 
$$ u^\eps(x) \: = \:  u^P(x) \: + \: \eps v^\eps(x/\eps)  $$ 
where $u^P$ is the reference Poiseuille flow, satisfying \eqref{eq:Stokes} with Dirichlet condition at both planes. Remind that
$$ u^P(x)  \: = \:  - x_3 (x_3-1) e.  $$
Hence, $v^\eps =( v^\eps_h(y), v^\eps_3(y))$ satisfies 
\begin{equation} 
\begin{aligned}
-\Delta v + \na p & = 0, \quad \mbox{in } \: \T^2 \times (0,\eps^{-1}), \\
 \div v & = 0, \quad \mbox{in } \: \T^2 \times (0,\eps^{-1}), \\
 v & = 0,  \quad y_3 = \eps^{-1}, \\
 v_3  & = 0,  \quad y_3 = 0,  \\
  v_h = 0,   \quad y \in \eps^{-1} T^\eps \times \{0\},  \quad \pa_{y_3} v_h & = - e, \quad   y \in \eps^{-1} (T^\eps)^{c} \times \{0\} \\
 \end{aligned}
 \end{equation}
 Note that no approximation has been made so far. It is then tempting to put the roof $y_3 = \eps^{-1}$ at infinity replacing $\T^2 \times (0,\eps^{-1})$ by $\T^2 \times \R_+$. However, it is  well-known that the resulting problem is overdetermined. Namely, the boundary layer field $v^{\eps,bl}$ satisfying 
 \begin{equation}  \label{BL}
\begin{aligned}
-\Delta v + \na p & = 0, \quad \mbox{in } \: \T^2 \times \R_+, \\
 \div v & = 0, \quad \mbox{in } \: \T^2 \times \R_+, \\
 v_3  & = 0,  \quad y_3 = 0,  \\
  v_h = 0,   \quad y \in \eps^{-1} T^\eps \times \{0\},  \quad \pa_{y_3} v_h & = - e, \quad   y \in \eps^{-1} (T^\eps)^{c} \times \{0\} \\
 \end{aligned}
 \end{equation}
 has constant horizontal average: 
 $$ v^{\eps,\infty}_h \: := \:   \int_{\T^2} v^{\eps,bl}_h(y) dy_1 dy_2 $$
 with respect to $y_3$. More precisely, it can be shown that 
 $$ v^{\eps,bl} \: \rightarrow \:  (v^{\eps,\infty}_h, 0) $$
 exponentially fast as $y_3$ goes to infinity. Furthermore, by linearity of \eqref{BL}, one may denote $v^{\eps,\infty}_h \: = \:  V^{\eps,\infty} \, e$ for a two by two matrix  $V^{\eps,\infty}$. Then,  one can show that $V^{\eps,\infty}$ is symmetric positive definite, with
 $$ V^{\eps,\infty} \, e \cdot e = \int_{\T^2 \times \R_+} |\na v^{\eps,bl}|^2. $$
Note that everything depends on $\eps$, through the rescaled domain $\eps^{-1} T^\eps$ in \eqref{BL}. 

\medskip
To correct the "boundary layer constant" at infinity, one must add a macroscopic Couette flow. One ends up with 
\begin{equation*}
 u^{\eps} \: \approx \:  u^P(x) \: + \: \eps v^{\eps,bl}(x/\eps) \: - \: \eps x_3 (V^{\eps,\infty} e, 0)  
\end{equation*}
Averaging in the small scale, we find 
$$ u^{\eps}_{h}\vert_{x_3 = 0}  \: \approx \: \eps V^{\eps,\infty} e, \quad \pa_3 u^\eps\vert_{x_3 = 0} \: \approx \: \pa_3 u^P\vert_{x_3 = 0}  \: \approx \:  e. $$ 
We end up with the approximate boundary condition 
\begin{equation} \label{approxBC}
u^\eps_h \: = \: \eps V^{\eps, \infty}  \pa_3 u^\eps_h \quad \mbox{ at } \: x_3 = 0.
\end{equation} 
 
 \medskip
On the basis of the previous reasoning, one can implement the following strategy for the numerical computation of the slip length: 
\begin{itemize}
\item Compute numerically (say with $e = e_1$ and $e=e_2$) the solution of \eqref{BL}, in order to determine the matrix $V^{\eps, \infty}$. 
\item Check for the asymptotics of $\eps V^{\eps,\infty}$, for various  shapes and sizes of the no-slip zone $T^\eps$.  This allows to make the comparison with theoretical results of Theorems 1 and 2. Indeed, sending $\eps$ to zero in  \eqref{approxBC} yields 
\be\label{Lim:epsVeps} 
\bar{u}_h \: = \: \lim_{\eps \rightarrow 0} \left(\eps V^{\eps, \infty} \right)  \pa_3 \bar{u}_h \quad \mbox{ at } \: x_3 = 0, 
\ee
so that the matrix $M$ in the theorems satisfies $M^{-1} =   \lim_{\eps \rightarrow 0} \left(\eps V^{\eps, \infty} \right)$. 
\end{itemize}

\paragraph{Numerical approximation of the matrix $\Vepsinf$.}

In the numerical simulations, we will solve the system \eqref{BL}
 associated to different shapes of the no-slip zone $T^\eps$: circular or rectangular patches, and riblets parallel or orthogonal to the flow. Let us first notice that for such configurations, the matrix $\Vepsinf$ is diagonal. Indeed, since the domain $\eps^{-1}T^\eps$ is symmetric with respect to the axis $\{y_2=1/2\}$, if we denote by $v$ the solution to system \eqref{BL} with $e=e_1$, then the vector field $v^*$ defined by  
$
v_i^*(y_1,y_2,y_3)=v_i(y_1,1-y_2,y_3)
$, for $i=1,3$, and by 
$
v_2^*(y_1,y_2,y_3)=-v_2(y_1,1-y_2,y_3),
$
 is also a solution. By uniqueness, we deduce that $v_2(y_1,1-y_2,y_3)=-v_2(y_1,y_2,y_3)$ for a.e.\,$(y_1,y_2,y_3)\in \T^2\times \R_+$, which yields $\Vepsinf e_1\cdot e_2=0$. By symmetry of $\Vepsinf$, we obtain also that $\Vepsinf e_2\cdot e_1=0$

Consequently  the boundary conditions satisfied by the  horizontal components of the approximate solution to system \eqref{eq:Stokes}-\eqref{noslip-z=1}-\eqref{mixed-z=0} on $x_3=0$, simply writes: 
\begin{equation} \label{approxBCi}
u^\eps_i \: = \: \eps (\Vepsinf e_i\cdot e_i)\  \pa_3 u^\eps_i \quad \mbox{ at } \: x_3 = 0,\quad \mbox{for }i=1,2.
\end{equation} 
In the rest of this section, for $i=1,2$, the quantity $\Vepsinf e_i\cdot e_i$ will be refered to as the \emph{average slip length} associated to our problem, in the direction $e_i$.

To compute an approximate value of the average slip length associated to system \eqref{BL}, we consider a truncated domain $\T^2\times (0,H)$, for a given $H>0$, and we introduce the solution $w$ to the following problem:
\begin{equation}  \label{BLapprox}
\begin{aligned}
-\Delta w + \na q & = 0, \quad \mbox{in } \: \T^2 \times (0,H), \\
 \div w & = 0, \quad \mbox{in } \: \T^2 \times (0,H), \\
 \d_{y_3} w - q\ e_3 & = 0,\quad y_3 = H,  \\
 w_3  & = 0,  \quad y_3 = 0,  \\
  w_h = 0,   \quad y \in \eps^{-1} T^\eps \times \{0\},  \quad \pa_{y_3} w_h & = - e, \quad   y \in \eps^{-1} (T^\eps)^{c} \times \{0\} \\
 \end{aligned}
 \end{equation}
 
 Using arguments developed in Ref. \refcite{JaMiNe}, the difference between $v^{\eps,bl}$ and $w$ can be estimated as follows. First, we claim that $v^{\eps,bl}$ satisfies the following $H^1$ bound:
 \be\label{borne-H1}
 \|\na v^{\eps,bl}\|_{L^2(\T^2\times \R_+)}\leq C \sqrt{\frac{\eps}{\aeps}},
 \ee
 where $C$ is a constant which does not depend on $\eps$. This bound follows from a quantitative trace inequality, whose proof is similar to the one of \eqref{ineq:super-critical}: there exists a constant $C>0$ such that for all $b_\eps\in (0,1)$, for all $v\in H^1(\T^2\times (0,1))$ such that $v_{|x_3=0}$ vanishes on a ball of radius $b_\eps$,
 $$
 \| v_{|x_3=0}\|_{L^2(\T^2)}\leq \frac{C}{\sqrt{b_\eps}} \| \na v\|_{L^2(\T^2\times (0,1))}.
 $$
 Then, we decompose $v^{\eps,bl}$ into horizontal Fourier series and we derive exponential decay bounds: for all $s\in \N$, there exists a constant $\gamma_s>0$, which does not depend  on $\eps$, such that
 \be\label{exp-bounds}
 \begin{aligned}
 \| v^{\eps,bl}(\cdot, x_3)- (V^{\eps,\infty} \, e ,0)\|_{L^2(\T^2)}\leq C \sqrt{\frac{\eps}{\aeps}}\exp(-\gamma_0 x_3),\\
 \sum_{\alpha\in\N^3,|\alpha|\leq s} \|\na^\alpha  v^{\eps,bl}(\cdot, x_3)\|_{L^2(\T^2)}\leq C \sqrt{\frac{\eps}{\aeps}}\exp(-\gamma_s x_3).
 \end{aligned}
 \ee
 As a consequence, $v^{\eps,bl}$ is a solution of \eqref{BLapprox} in $\T^2\times(0,H)$, with a slightly modified condition at $y_3=H$, namely
 $$\begin{aligned}
 \d_{y_3}v^{\eps,bl} - p^{\eps,bl}\ e_3 =F^\eps \quad \text{at }y_3=H,\\
 \text{with}\quad \|F^\eps\|_{H^s(\T^2)}\leq C \sqrt{\frac{\eps}{\aeps}}\exp(-\gamma_s H) \ \forall s\in \N\quad \text{and}\quad \int_{\T^2} F^\eps=0.
 \end{aligned}
$$
 It follows that there exist constants $C,\gamma>0$ such that
 $$
 \|\na (v^{\eps,bl}- w)\|_{L^2(\T^2\times (0,H))}^2\leq C\frac{\eps}{\aeps} \exp(-\gamma H).
 $$
 Notice also that $w$, as $v^{\eps,bl}$, has constant horizontal average and that 
 $$
 \int_0^H\int_{\T^2}|\na w|^2=\int_{\T^2} w (y) \:dy_1\:dy_2.
 $$
 
We solve problem \eqref{BLapprox} by a finite element method. We use $P_2$ elements for the velocity and $P_1$ elements for the pressure. The three-dimensional mesh of the fluid domain $\T^2\times (0,H)$ is obtained by a constrained Delaunay tetrahedralization. 
The incompressibility condition is treated by a Lagrange multiplier (see Ref. \refcite{GiraultRaviart1986}, Ref. \refcite{ItoKunisch2008}).

Given two approximate solutions $w^1_{app}, w^2_{app}$ of system \eqref{BLapprox}, associated respectively to $e=e_1$ and $e=e_2$,  we define the numerical approximation $\Vepsinf_{app}$ of the matrix $\Vepsinf$, by the following formula: 
\bestar
\Vepsinf_{app} e_i\cdot e_j:=\int_{\T^2} w^i_{\textit{app}}(y_1,y_2,H)\cdot e_j\ d y_1 d y_2,\quad \mbox{for }i,j=1,2.
\eestar
By analogy with formula \eqref{approxBCi}, for $i=1,2$, the \emph{approximate average slip length} in direction $e_i$ is then defined by $ \Vepsinf_{app} e_i\cdot e_i $.

Finally,  we  introduce the \emph{solid fraction} $\phiseps$, which is defined by the relative area of the no-slip zone $T^\eps$ in the elementary square of size $\eps$ (or equivalently, by the area of the rescaled no-slip domain $\eps^{-1}T^\eps$). Using definitions \eqref{DefPatches}-\eqref{DefRiblets}, $\phiseps$ is given by the following expressions:
\begin{itemize}
\item in the case of patches, $\disp\phiseps=\left(\frac{\aeps}{\eps}\right)^2|T|$, where $|T|$ stands for the area of the domain $T$;
\item in the case of riblets, $\disp\phiseps=\frac{\aeps}{\eps}|I|$, where $|I|$ stands for the length of the interval $I$.
\end{itemize}
Notice that system \eqref{BL} is completely determined by $\phiseps$  and by the domain $T$ (in the case of patches) or the union of intervals $I$ (in the case of riblets).
%The quantity $\phiseps$ will be refered to as the \emph{solid fraction}, in the sense that it represents the relative area of the solid contact zone $T^\eps$ in the elementary square of size $\eps$.

\paragraph{Computation of the average slip length, in the case of patches.}

In the case of patches, we have plotted $ \Vepsinf_{app} e_1\cdot e_1 $ against $1/\sqrt{\phiseps}$, considering circular and squared patches (see Figure \ref{Fig:patches-disks-squares}). We observe that the dependency is affine, and a linear regression gives the relation $\Vepsinf_{app} e_1\cdot e_1 \approx \alpha/\sqrt{\phiseps} + \beta$, with $\alpha=0.322$, $\beta = - 0.429 $ in the case of the disk, and $\alpha = 0.311$, $\beta = - 0.422 $ in the case of the square. Note that these coefficients are very close to the ones obtained by Ybert \emph{et al.} in Ref. \refcite{Ybert}. Consequently, since $\lim_{\eps\ra 0} \phiseps = 0$, 
\be\label{Scaling:patchnum}
\Vepsinf_{app} e_1\cdot e_1 \sim \frac{\alpha}{\sqrt{\phiseps}}\quad \text{as }\eps \ra 0.
\ee

To compare this numerical result with the theoretical result given by Theorem \ref{thm:patch}, let us consider the critical case $\aeps/\eps^2 \ra  C_0>0$. In that case, there exists a two by two matrix $M_0$, depending on the pattern $T$, such that $\lim_{\eps\ra 0}\eps \Vepsinf = \frac{1}{C_0} M_0^{-1}$. For circular or squared patterns centered in the unit square, as observed above, the matrices $\Vepsinf$, and consequently the matrix $M_0$, are diagonal. Moreover, since these patterns are invariant by a rotation of angle $\pi/2$, one can easily see that the corresponding matrix $\Vepsinf$ satisfies $\Vepsinf e_1\cdot e_1 = \Vepsinf e_2\cdot e_2$. Consequently, there exists $\lambda_0>0$ such that $M_0=\left( \begin{smallmatrix}
\lambda_0 & 0\\
0 & \lambda_0\\
\end{smallmatrix}\right)$, and the following relation holds:
\bestar
\lim_{\eps\ra 0} \eps \Vepsinf e_1\cdot e_1 = \frac{1}{C_0 \,\lambda_0}.
\eestar
Besides, using the definition of $\phiseps$ in the case of patches, the asymptotic relation \eqref{Scaling:patchnum} yields
\bestar
\lim_{\eps\ra 0} \eps \Vepsinf_{app} e_1\cdot e_1 = \frac{\alpha}{C_0\sqrt{|T|}}.
\eestar
Thus, the numerical value of the slip length $\alpha/(C_0\sqrt{|T|})$, that can be deduced from the asymptotic behavior \eqref{Scaling:patchnum} in the critical case, is consistent with Theorem \ref{thm:patch}. The coefficient of the matrix $M_0$ can be approximated by $\lambda_0 \approx \sqrt{|T|}/\alpha$.

We notice that the results concerning the sub-critical and super-critical cases can also be retrieved, at least formally, from relation \eqref{Scaling:patchnum}. Indeed, since $\eps/\sqrt{\phiseps}=\eps^2/(\aeps \sqrt{|T|})$, we obtain in the sub-critical case:
$
\lim_{\eps\ra 0} \eps \Vepsinf_{app} e_1\cdot e_1 = +\infty,
$
which corresponds formally to an infinite slip length in the $e_1$ direction, that is, a perfect slip condition. In the same manner, in the super-critical case, we obtain 
$
\lim_{\eps\ra 0} \eps \Vepsinf_{app} e_1\cdot e_1 = 0,
$
which corresponds to adherence in the $e_1$ direction.

\paragraph{Computation of the average slip length, in the case of riblets.}

In that case, exact computations are available in the literature, that give the average slip lengths in the $e_1$ and $e_2$ direction as a function of the solid fraction $\phiseps$ (see for instance Ref. \refcite{Phi}):
\be \label{ExactFormulaRib}
\Vepsinf e_1\cdot e_1 = -\ln\left[\cos\left(\frac{\pi}{2}(1-\phiseps)\right) \right]/\pi,\quad
\Vepsinf e_2\cdot e_2 = -\ln\left[\cos\left(\frac{\pi}{2}(1-\phiseps)\right) \right]/(2\pi). 
\ee
We have plotted in Figure \ref{Fig:riblets-x-y} the computed value of the average slip lengths $\Vepsinf_{app} e_1\cdot e_1$ and $\Vepsinf_{app} e_2\cdot e_2$, against $\phiseps$, as well as the exact values defined by formulas \eqref{ExactFormulaRib}. We observe that the numerical values are close to the expected ones. %However, we notice a certain loss of precision for small values of $\phiseps$. This phenomenon can be explained by the fact that for very thin domains $\eps^{-1}T^\eps$, 

Once again, formulas \eqref{ExactFormulaRib} and the numerical behavior of the average slip length shown in Figure \ref{Fig:riblets-x-y}, are consistent with the theoretical results of Theorem \ref{thm:riblet}. Indeed, in the critical case $\lim_{\eps\ra 0} -\eps \ln (\aeps)=C_0>0$, using the expression $\phiseps=(\aeps|I|)/\eps$, one obtains by a straightforward computation that $\eps \ln\left[\cos\left(\frac{\pi}{2}(1-\phiseps)\right) \right]\ra -C_0$ as $\eps \ra$. Consequently, the slip length in the directions $e_1$ and $e_2$ are respectively given by
\bestar
\lim_{\eps\ra 0}\eps \Vepsinf e_1\cdot e_1 = \frac{C_0}{\pi},\quad \lim_{\eps\ra 0}\eps \Vepsinf e_2\cdot e_2 = \frac{C_0}{2\pi}.
\eestar

\paragraph{Influence of the shape of the no-slip area: comparative results.}

In order to provide a comparison between the efficiency of patches and riblets in terms of slip length, we consider the slip length in the direction of the constant pressure gradient $f=2e_i$, with $i=1$ or $i=2$.  For circular or squared patterns, the average slip length is given by $\Vepsinf e_1\cdot e_1$. In the case of riblets, we consider two configurations of physical interest:
\begin{itemize}
\item riblets parallel to the flow: $f=2e_1$, the average slip length is defined by $\Vepsinf e_1\cdot e_1$;
\item riblets orthogonal to the flow: $f=2e_2$, the average slip length is $\Vepsinf e_2\cdot e_2$.
\end{itemize}
The results are plotted in Figure \ref{Fig:trace-global}. As stated in Remark \ref{rem:patches-vs-riblets}, page \pageref{rem:patches-vs-riblets}, these numerical results confirm that the riblets parallel to the flow are not necessarily optimal. Indeed, if the solid fraction $\phiseps$ is small enough, say $\phiseps<0.1$, the circular or squared patches produce a superior slip length.

To estimate the influence of the shape of the pattern on the slip length, we have considered families of rectangles of fixed area $\phiseps$, that are centered in the unit square. For $\phiseps=0.01, 0.04, 0.09$ we have computed the average slip length $\Vepsinf_{app} e_1\cdot e_1$, in the direction $e_1$, associated to each of these rectangular patterns. The results are plotted in figure \ref{Fig:allongement}, against the dimension $L$ of each rectangular pattern, in the $e_1$ direction. For each solid fraction $\phiseps$, the extremal values associated to $L=\phiseps$ and $L=1$, correspond respectively to a riblet orthogonal to the flow, and parallel to the flow.

We notice that, for each family of rectangular patterns of fixed area, the riblet orthogonal to the flow provides always the smallest average slip length. As already mentionned, the riblet parallel to the flow is not optimal, especially for small values of the solid fraction $\phiseps=0.01$, $\phiseps=0.04$. In that cases, the curves present  a unique maximum, and the associated optimal size $L$ of the rectangle is slightly superior to the size $\sqrt{\phiseps}$ of the square of same area. For these values of the solid fraction, the optimal rectangular pattern will present a certain anisotropy in the direction of the flow.

%As mentionned in Remark \eqref{rem:patches-vs-riblets},

%In order to compare the efficiency of circular patterns, squared patterns, and riblets, we have plotted in Figure \ref{Fig:trace-global} the numerical value of the average slip length in the direction of the flow

%%%%%%%%%%%%%%%%%%%%%%%%%%%%%%%%%%%%%%%%%%%%%%%%
%%%%%%%%%%%%%%%%%%% Graphiques %%%%%%%%%%%%%%%%%%%%
%%%%%%%%%%%%%%%%%%%%%%%%%%%%%%%%%%%%%%%%%%%%%%%%

\begin{figure}
\begin{center}
\begin{tikzpicture}[scale=0.8]
\begin{axis}[ axis x line=bottom, axis y line = left,
xlabel={$1/\sqrt{\phiseps}$}, ylabel={$\Vepsinf_{app} e_1\cdot e_1$},
ymin=0, ymax=2.5,
legend entries={Circular patches,Squared patches}, legend style={at={(0.02,1)}, anchor=north west}]
\addplot[mark=o] table[x index=3, y index=4]{res_disque-light.txt};
\addplot[mark=square] table[x index =3, y index=4]{res_carre-light.txt};
\end{axis}
\end{tikzpicture}
\end{center}
\caption{Numerical value of the average slip length $\Vepsinf_{app} e_1\cdot e_1$ plotted against $1/\sqrt{\phiseps}$, for circular patches and squared patches.}
\label{Fig:patches-disks-squares}
\end{figure}
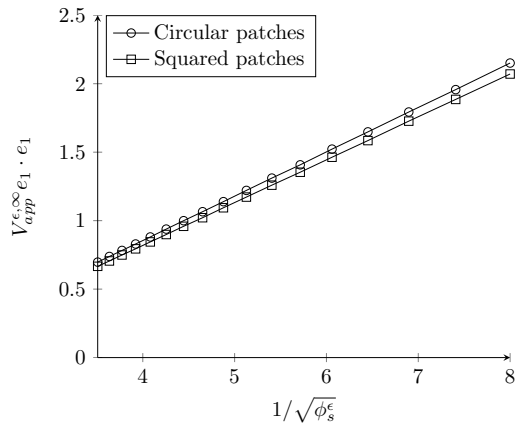

\begin{figure}
%\begin{minipage}[t]{0.5\textwidth}
\begin{center}
\begin{tikzpicture}[scale=0.8]
\begin{axis}[ axis x line=bottom, axis y line = left,
xlabel={$\phiseps$}, ylabel={Average slip length},
ymin=0,
legend entries={$\Vepsinf_{app} e_1\cdot e_1$, $\Vepsinf_{app} e_2\cdot e_2$, Exact values}, 
legend style={at={(1,1)}, anchor=north east}]
\addplot[mark=-] table[x index=1, y index=4]{res_riblet_x-global.txt};
\addplot[mark=|] table[x index=1, y index=4]{res_riblet_y-global.txt};
\addplot[mark=no markers,dashed] table[x index=1, y index=5]{res_riblet_x-complet.txt};
\addplot[mark=no markers,dashed] table[x index=1, y index=5]{res_riblet_y-complet.txt};
\end{axis}
\end{tikzpicture}
\end{center}
\caption{Numerical values of the average slip lengths $\Vepsinf_{app} e_1\cdot e_1$ and $\Vepsinf_{app} e_2\cdot e_2$, plotted against $\phiseps$, in the case of riblets. The dashed lines represent the exact value of the average slip lengths, defined by formulas \eqref{ExactFormulaRib}.}
\label{Fig:riblets-x-y}
\end{figure}
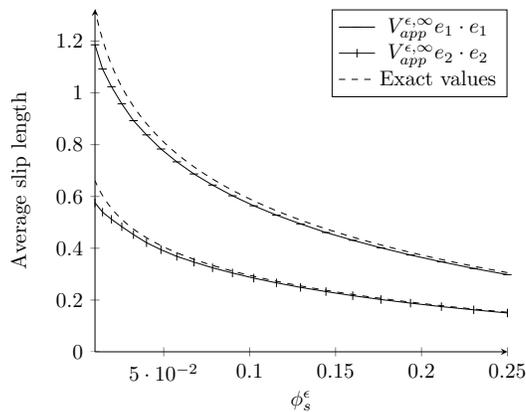

\begin{figure}
\begin{center}
\begin{tikzpicture}[scale=0.8]
\begin{axis}[ axis x line=bottom, axis y line = left,
xlabel={$\phiseps$}, ylabel={Average slip length},
legend entries={Disks, Squares, Riblets $\parallel$, Riblets $\bot$}, 
legend style={at={(1,1)}, anchor=north east}]
\addplot[mark=o] table[x index=1, y index=4]{res_disque-global.txt};
\addplot[mark=square] table[x index =1, y index=4]{res_carre-global.txt};
\addplot[mark=-] table[x index=1, y index=4]{res_riblet_x-global.txt};
\addplot[mark=|] table[x index=1, y index=4]{res_riblet_y-global.txt};
\end{axis}
\end{tikzpicture}
\end{center}
\caption{Average slip length in the direction of the flow, in the case of circular patches, squared patches and riblets parallel and orthogonal to the flow.}
\label{Fig:trace-global}
\end{figure}
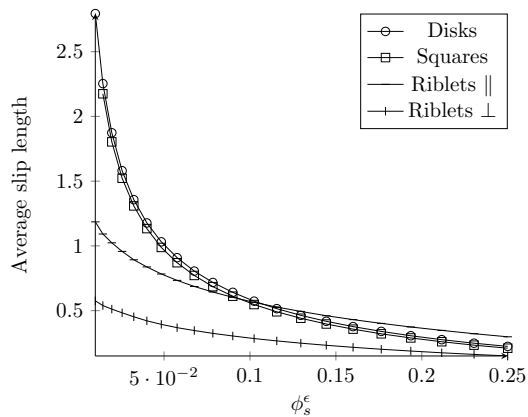

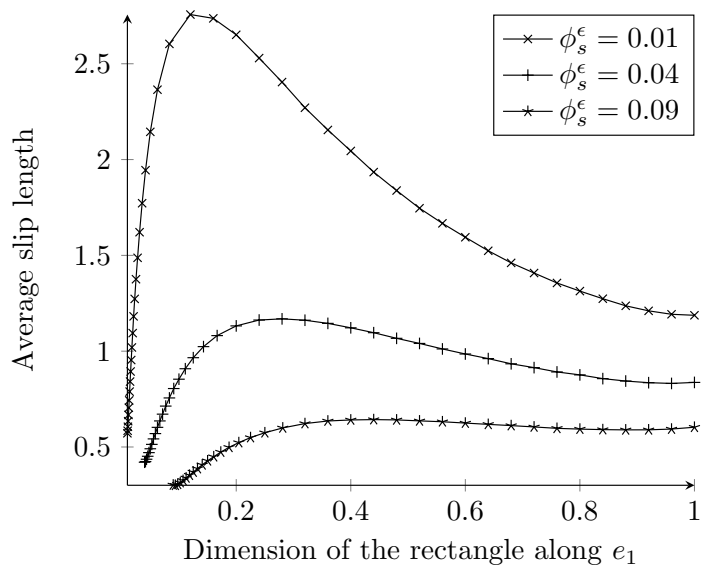
\begin{figure}
\begin{center}
\begin{tikzpicture}[scale=1.1]
\begin{axis}[ axis x line=bottom, axis y line = left,
xlabel={Dimension of the rectangle along $e_1$}, ylabel={Average slip length},
legend entries={$\phiseps=0.01$,$\phiseps=0.04$,$\phiseps=0.09$}, 
legend style={at={(1,1)}, anchor=north east}]
\addplot[mark=x] table[x index=1, y index=5]{allongement-trace-0.01.txt};
%\addplot[mark=+] table[x index =1, y index=5]
%{allongement-trace-0.0256.txt};
\addplot[mark=+] table[x index=1, y index=5]{allongement-trace-0.04.txt};
\addplot[mark=star] table[x index=1, y index=5]{allongement-trace-0.09.txt};
\end{axis}
\end{tikzpicture}
\end{center}
\caption{Numerical values of the average slip length $\Vepsinf_{app}e_1\cdot e_1$, produced by rectangular patterns of given area $\phiseps\in\{0.01, 0.04, 0.09\}$, and plotted against the dimension $L$ of the rectangle in the $e_1$ direction. The extremal value $L=\phiseps$ (resp.\,$L=1$) corresponds to the riblet  orthogonal to the flow (resp.\,parallel to the flow).}
\label{Fig:allongement}
\end{figure}

%%%%%%%%%%%%%%%%%%%%%%%%%%%%%%%
%%%%%%Appendice %%%%%%%%
%%%%%%%%%%%%%%%%%%%%%%%%%%%%%%

\section*{Appendix: proof of inequality \eqref{ineq:super-critical}}

To obtain inequality \eqref{ineq:super-critical}, it is enough to prove that for every $k\in [|0, \eps^{-1}|]^2$, 
\be \label{sucritique:ineq-sur-cellule}
\int_{S^\eps_k \times\{0\}} |u^{\eps}|^2 \: \leq \:  \eta(\eps)\int_{B_k^{\eps, +}}|\nabla u^{\eps}|^2\qquad \forall 0<\eps<\eps_0.
\ee
A summation over $k\in [|0,\eps^{-1}|]^2$ then leads to inequality \eqref{ineq:super-critical}.%, where $\eta(\eps)$ is replaced by $9\eta(\eps)$

\medskip
Let $k\in  [|0, \eps^{-1}|]^2$. By rescaling the trace inequality in the half cube $[0,1]^2\times [0,\frac{1}{2}]$, we obtain the existence of a constant $C>0$ such that
\be \label{ineq:trace-rescaled}
\int_{S_k^\eps\times \{0\}}|u^\eps|^2\leq C \left(\eps \int_{P_k^{\eps,+}}|\nabla u^\eps|^2 + \frac{1}{\eps} \int_{P_k^{\eps,+}}|u^\eps|^2\right).
\ee
To estimate the $L^2$ norm of $u^\eps$ by the $L^2$ norm of its gradient, we adapt Lemma 3.4.1 in Ref.  \refcite{Allaire2} to our bidimensional array of holes. We denote by $B^\eps_k$ the ball circumscribing the cube $P_k^{\eps}$. Of course, the upper half-cube  $P_k^{\eps, +}$ is contained in the upper half-ball $B_k^{\eps,+}$. Moreover, since the model no-slip zone $T$ contains a disk of radius $\alpha$ centered at the origin, each elementary no-slip pattern $\eps k + T^\eps$ contains a disk  of radius $\aeps\alpha$, centered in the square $S_k^\eps$. Let $\tB_k^\eps$ be the 3d ball of same center and radius, and $\tB_k^{\eps,+}$ be the corresponding half ball. With this notation, we can write 
\bestar
\int_{P_k^{\eps,+}}|u^\eps|^2 \leq \int_{B_k^{\eps,+}\setminus \tB_k^{\eps,+}}|u^\eps|^2 + \int_{\tB_k^{\eps,+}}|u^\eps|^2.
\eestar
To estimate the contribution of the exterior part $B_k^{\eps,+}\setminus \tB_k^{\eps,+}$, we use spherical coordinates $(\rho,\phi,\theta)$ centered at point $\: \eps k + (\frac{\eps}{2}, \frac{\eps}{2}, 0)$. 
%We still denote by $u^\eps(\rho,\phi,\theta)$ the value of $u^\eps$ at point $c_k^\eps + \rho(\cos\phi \sin\theta, \sin\phi \sin \theta,\cos\theta )$. 
The radius of $B_k^\eps$ being equal to $\frac{\eps\sqrt{3}}{2}$, integrating along rays, we get for every $r',r$ such that $0<r' < \aeps \alpha < r < \frac{\eps\sqrt{3}}{2}$,
\bestar
u^\eps(r,\phi,\theta)=u^\eps(r',\phi,\theta ) + \int_{r'}^r \d_{\rho} u^\eps(\rho,\phi,\theta ) \ud \rho,
\eestar
which yields
\bestar
|u^\eps(r,\phi,\theta)|^2 \leq 2 |u^\eps(r',\phi,\theta)|^2 + 2 \left( \int_{r'}^r \d_{\rho} u^\eps(\rho,\phi,\theta) \ud \rho \right)^2. 
\eestar
Multiplying last inequality by $r^2(r')^2\sin \theta$ and integrating on $r'\in (0,\aeps \alpha)$, $r\in(\aeps\alpha, \frac{\eps\sqrt{3}}{2})$, $\phi\in (0,2\pi)$, $\theta\in (0,\pi/2)$, we obtain the inequality
\be \label{ineq:IJK} 
I^\eps\leq 2 J^\eps + 2 K^\eps
\ee
where the integrals $I^\eps$, $J^\eps$, $K^\eps$ are respectively defined by
\begin{align*} 
I^\eps & =\int_{r'=0}^{\aeps\alpha }\int_{r=\aeps\alpha}^{\frac{\eps\sqrt{3}}{2}}\int_{\theta}\int_{\phi} |u(r,\phi,\theta)|^2 r^2 (r')^2 \sin \theta\ \ud \theta\ \ud \phi\ \ud r\ \ud r',\\
J^\eps & = \int_{r'=0}^{\aeps\alpha}\int_{r=\aeps\alpha}^
{\frac{\eps\sqrt{3}}{2}}\int_{\theta}\int_{\phi} |u(r',\phi,\theta)|^2 r^2 (r')^2 \sin \theta\ \ud \theta\ \ud \phi\ \ud r\ \ud r', \\
K^\eps & =\int_{r'=0}^{\aeps\alpha}\int_{r=\aeps\alpha}^
{\frac{\eps\sqrt{3}}{2}}\int_{\theta}\int_{\phi} \left( \int_{r'}^r \d_{\rho} u^\eps(\rho,\phi,\theta) \ud \rho \right)^2 r^2 (r')^2 \sin \theta\ \ud \theta\ \ud \phi\ \ud r\ \ud r'.
\end{align*}
By Fubini theorem,
$$
I^\eps  = \left( 
\int_{0}^{\aeps\alpha }\ (r')^2\ud r'
\right)
\left( 
\int_{r=\aeps\alpha}^{\frac{\eps\sqrt{3}}{2}}\int_{\theta}\int_{\phi} |u(r,\phi,\theta)|^2 r^2  \sin \theta\ \ud \theta\ \ud \phi\ \ud r
\right)\\
 = \frac{\aeps^3 \alpha^3 }{3}\int_{B_k^{\eps,+}\setminus \tB_k^{\eps,+}}|u^\eps|^2,
$$
and by an analogous computation,
\begin{align*}
J^\eps & =\left(\eps^3\frac{\sqrt{3}}{8} -  \frac{\aeps^3\alpha^3}{3} \right)\int_{\tB_k^{\eps,+}}|u^\eps|^2.
\end{align*}
By Schwarz inequality, 
$$
\left( \int_{r'}^r \d_{\rho} u^\eps \ud \rho \right)^2 \leq \left( \int_{r'}^r \frac{1}{\rho^2}\ud \rho\right) \left( \int_{r'}^r \rho^2 |\d_{\rho} u^\eps|^2  \ud \rho \right)\\
 \leq \frac{1}{r'} \int_{r'}^r \rho^2 |\d_{\rho} u^\eps|^2  \ud \rho,
$$
which yields
$$
K^\eps\leq \left( \int_{\aeps\alpha}^{\frac{\eps\sqrt{3}}{2}}r^2\ \ud r\right) \left( \int_0^{\aeps\alpha}r'\ \ud r'\right)\left(\int_{B_k^{\eps,+}}|\nabla u^\eps|^2\right)\\
\leq \frac{\sqrt{3}}{16}\aeps^2\alpha^2\eps^3 \int_{B_k^{\eps,+}}|\nabla u^\eps|^2.
$$
Consequently, inequality \eqref{ineq:IJK} leads to
\be \label{ineq:surcritique-couronne1}
\int_{B_k^{\eps,+}\setminus \tB_k^{\eps,+}}|u^\eps|^2 \leq \frac{3\sqrt{3}}{8}\frac{\eps^3}{\aeps\alpha} \left( \frac{2}{\aeps^2\alpha^2} \int_{\tB_k^{\eps,+}}|u^\eps|^2 + \int_{B_k^{\eps,+}}|\nabla u^\eps|^2\right).
\ee
Since $u^\eps$ vanishes on $\tB_k^{\eps,+}\cap (\R^2\times\{0\})$, using Poincar\'e inequality in a cylinder of height $\aeps\alpha$, we obtain  the following estimate:
\begin{align*}
\int_{\tB_k^{\eps,+}}|u^\eps|^2 
& \leq \aeps^2\alpha^2\int_{B_k^{\eps,+}}|\nabla u^\eps|^2.
%\\
%& \leq \aeps^2\alpha^2\int_{B(\aeps \cb,\sqrt{3}\eps)}|\nabla u^\eps|^2. 
\end{align*}
Injecting this inequality into estimate \eqref{ineq:surcritique-couronne1}, we obtain:
\bestar
\int_{B_k^{\eps,+}\setminus \tB_k^{\eps,+}}|u^\eps|^2 \leq \frac{9\sqrt{3}\eps^3}{\aeps\alpha} \int_{B_k^{\eps,+}}|\nabla u^\eps|^2,
\eestar
and summing these two inequalities, we get
\bestar
\int_{B_k^{\eps,+}} |u^\eps|^2 \leq \left( \aeps^2\alpha^2 + \frac{9\sqrt{3}\eps^3}{\aeps\alpha} \right) \int_{B_k^{\eps,+}}|\nabla u^\eps|^2.
\eestar
Finally, using inequality \eqref{ineq:trace-rescaled}, we obtain estimate \eqref{sucritique:ineq-sur-cellule}, where $\eta(\eps)$ is defined by 
\bestar
\eta(\eps)=C\left(\eps + \frac{\aeps^2\alpha^2}{\eps} + \frac{9\sqrt{3}\eps^2}{\aeps\alpha} \right),
\eestar
and converges to $0$ as $\eps\ra 0$, since $\aeps<\eps$ and $\aeps\gg \eps^2$.
\qed
\bibliography{BDGV4}

\providecommand{\bysame}{\leavevmode\hbox to3em{\hrulefill}\thinspace}
\providecommand{\MR}{\relax\ifhmode\unskip\space\fi MR }
\providecommand{\MRhref}[2]{%
  \href{http://www.ams.org/mathscinet-getitem?mr=#1}{#2}
}
\providecommand{\href}[2]{#2}
\begin{thebibliography}{10}

\bibitem{AcPiVa}
Yves Achdou, O.~Pironneau, and F.~Valentin, \emph{Effective boundary conditions
  for laminar flows over periodic rough boundaries}, J. Comput. Phys.
  \textbf{147} (1998), no.~1, 187--218. \MR{1657773 (99j:76086)}

\bibitem{Allaire1}
Gr{\'e}goire Allaire, \emph{Homogenization of the {N}avier-{S}tokes equations
  in open sets perforated with tiny holes. {I}. {A}bstract framework, a volume
  distribution of holes}, Arch. Rational Mech. Anal. \textbf{113} (1990),
  no.~3, 209--259.

\bibitem{Allaire2}
\bysame, \emph{Homogenization of the {N}avier-{S}tokes equations in open sets
  perforated with tiny holes. {II}. {N}oncritical sizes of the holes for a
  volume distribution and a surface distribution of holes}, Arch. Rational
  Mech. Anal. \textbf{113} (1990), no.~3, 261--298.

\bibitem{AllaireSlip91}
Gr\'egoire Allaire, \emph{Homogenization of the navier-stokes equations with a
  slip boundary condition}, Communications on Pure and Applied Mathematics
  \textbf{44} (1991), no.~6, 605--641.

\bibitem{AmBrLe}
Youcef Amirat, Didier Bresch, J{\'e}r{\^o}me Lemoine, and Jacques Simon,
  \emph{Effect of rugosity on a flow governed by stationary {N}avier-{S}tokes
  equations}, Quart. Appl. Math. \textbf{59} (2001), no.~4, 769--785.
  \MR{1866556 (2002g:76036)}

\bibitem{AmSi}
Youcef Amirat and Jacques Simon, \emph{Influence de la rugosit\'e en
  hydrodynamique laminaire}, C. R. Acad. Sci. Paris S\'er. I Math. \textbf{323}
  (1996), no.~3, 313--318. \MR{1404780 (97f:76025)}

\bibitem{BaGe}
Arnaud Basson and David G{\'e}rard-Varet, \emph{Wall laws for fluid flows at a
  boundary with random roughness}, Comm. Pure Appl. Math. \textbf{61} (2008),
  no.~7, 941--987. \MR{2410410 (2009h:76055)}

\bibitem{BuDaGe}
D.~Bucur, A.-L. Dalibard, and D.~Gerard-Varet, \emph{Wall laws for viscous
  fluids near rough surfaces}, ESAIM Proc \textbf{37} (2012), 117--135.

\bibitem{terme-etrange}
D.~Cioranescu and F.~Murat, \emph{Un terme \'etrange venu d'ailleurs},
  Nonlinear partial differential equations and their applications. {C}oll\`ege
  de {F}rance {S}eminar, {V}ol. {II} ({P}aris, 1979/1980), Res. Notes in Math.,
  vol.~60, Pitman, Boston, Mass., 1982, pp.~98--138, 389--390. \MR{652509
  (84e:35039a)}

\bibitem{DaGe}
Anne-Laure Dalibard and David G{\'e}rard-Varet, \emph{Effective boundary
  condition at a rough surface starting from a slip condition}, J. Differential
  Equations \textbf{251} (2011), no.~12, 3450--3487. \MR{2837691 (2012k:35031)}

\bibitem{Galdi}
G.~P. Galdi, \emph{An introduction to the mathematical theory of the
  {N}avier-{S}tokes equations}, second ed., Springer Monographs in Mathematics,
  Springer, New York, 2011, Steady-state problems. \MR{2808162 (2012g:35233)}

\bibitem{GiraultRaviart1986}
Vivette Girault and Pierre-Arnaud Raviart, \emph{Finite element methods for
  {N}avier-{S}tokes equations}, Springer Series in Computational Mathematics,
  vol.~5, Springer-Verlag, Berlin, 1986, Theory and algorithms. \MR{851383
  (88b:65129)}

\bibitem{ItoKunisch2008}
Kazufumi Ito and Karl Kunisch, \emph{Lagrange multiplier approach to
  variational problems and applications}, Advances in Design and Control,
  vol.~15, Society for Industrial and Applied Mathematics (SIAM), Philadelphia,
  PA, 2008. \MR{2441683 (2009g:49001)}

\bibitem{JaMi}
Willi J{\"a}ger and Andro Mikeli{\'c}, \emph{On the roughness-induced effective
  boundary conditions for an incompressible viscous flow}, J. Differential
  Equations \textbf{170} (2001), no.~1, 96--122. \MR{1813101 (2002b:76049)}

\bibitem{JaMiNe}
Willi J{\"a}ger, Andro Mikeli{\'c}, and Nicolas Neuss, \emph{Asymptotic
  analysis of the laminar viscous flow over a porous bed}, SIAM J. Sci. Comput.
  \textbf{22} (2000), no.~6, 2006--2028 (electronic).

\bibitem{Lauga}
E.~Lauga, M.P. Brenner, and H.A. Stone, \emph{Microfluidics: The no-slip
  boundary condition}, Handbook of Experimental Fluid Dynamics,C. Tropea, A.
  Yarin, J. F. Foss (Eds.), Springer, 2007, p.~123601.

\bibitem{Phi}
John~R. Philip, \emph{Flows satisfying mixed no-slip and no-shear conditions},
  Z. Angew Math. Phys. \textbf{23} (1972), 353--372. \MR{0321415 (47 \#9948)}

\bibitem{Solonnikov-Scadilov}
V.~A. Solonnikov and V.~E. \v{S}\v{c}adilov, \emph{On a boundary value problem
  for a stationary system of {N}avier-{S}tokes equations}, Proc. Steklov Inst.
  Math. \textbf{125} (1973), 186--199.

\bibitem{Vino}
O.~Vinogradova and G.~Yakubov, \emph{Surface roughness and hydrodynamic
  boundary conditions}, Phys. Rev. E (2006), 479--487.

\bibitem{Ybert}
C.~Ybert, C.~Barentin, C.~Cottin-Bizonne, P.~Joseph, and Bocquet L.,
  \emph{Achieving large slip with superhydrophobic surfaces: Scaling laws for
  generic geometries}, Physics of fluids \textbf{19} (2007), 123601.

\end{thebibliography}
\end{document}